\documentclass[11pt,a4paper]{article}

\usepackage[a4paper,top=3.8cm,bottom=3.8cm,left=3cm,right=3cm]{geometry}

\usepackage[T1]{fontenc}
\usepackage{amsmath,amssymb,mathtools,amsthm,amsopn}
\usepackage{hyperref}

\usepackage{titlesec}
\titleformat{\section}{%
\normalfont\large\bfseries}{\thesection.}{1em}{}
\titleformat{\subsection}{%
\normalfont\normalsize\bfseries}{\thesubsection.}{1em}{}

\usepackage{enumitem}

\usepackage{tocloft}
\usepackage{fancyhdr}
\usepackage{fourier-orns}

\usepackage{cite}

\renewcommand{\phi}{\varphi}

\newcommand{\U}{\mathcal{U}}

\newcommand{\I}{\mathcal{I}}
\newcommand{\step}[1]{\par\medskip\noindent\it#1\rm}

\newcommand{\dcc}{d_{\textup{cc}}}
\DeclareMathOperator{\espo}{e}
\newcommand{\eap}{\espo_{\textup{ap}}}
\newcommand{\expap}{\exp_{\textup{ap}}}

\newcommand{\abs}[1]{\lvert#1\rvert}

\newcommand{\g}{\gamma}

\newcommand{\wh}{\widehat}
\renewcommand{\H}{\mathcal{H}}
\newcommand{\e}{\varepsilon}
\renewcommand{\r}{\varrho}
\newcommand{\B}{\mathcal{B}}

\newcommand{\s}{\sigma}
\newcommand{\cc}{\textup{cc}}
\renewcommand{\P}{\mathcal{P}}

\renewcommand{\O}{\mathcal{O}}

\renewcommand{\cal}[1]{\mathcal{#1}}
\newcommand{\wt}{\widetilde}

\newcommand{\ol}{\overline}

\newcommand{\A}{\mathcal{A}}
\renewcommand{\d}{\delta}
\newcommand{\Eucl}{\textup{Euc}}
 
\newcommand{\p}{\partial}

\newcommand{\R}{\mathbb{R}}
\newcommand{\N}{\mathbb{N}}

\newtheoremstyle{pippo}  
  {}       
  {}       
   {\sffamily}   
 {}        
  {\bfseries}  
  {.}   
  {1ex}       
  {}           
\newtheoremstyle{pluto}  {}{}
{\slshape}  {}{\bfseries}  {.} {1ex}    {}

\newtheorem{theorem}{Theorem}[section]
\newtheorem{proposition}[theorem]{Proposition}

\newtheorem{lemma}[theorem]{Lemma}

\newtheorem{corollary}[theorem]{Corollary}

\theoremstyle{pluto} 
\newtheorem{definition}[theorem]{Definition}
\newtheorem{remark}[theorem]{Remark}

\renewcommand{\d}{\delta}
\renewcommand{\t}{\tau}

\renewcommand{\a}{\alpha}
\renewcommand{\b}{\beta}

\DeclareMathOperator{\Lip}{Lip}
\DeclareMathOperator{\Span}{span}

\newcommand{\norm}[1]{\left\Vert#1\right\Vert}
\numberwithin{equation}{section}

\let\oldbibliography\thebibliography
\renewcommand{\thebibliography}[1]{%
  \oldbibliography{#1}%
  \setlength{\itemsep}{0pt}%
}

\newenvironment{enumerate*}{\begin{enumerate}[noitemsep] }{\end{enumerate}}
\newenvironment{itemize*}{\begin{itemize}[noitemsep] }{\end{itemize}}
\newenvironment{description**}{\begin{description}[noitemsep] }{\end{description}}

\begin{document}

\title{Step-$s$ involutive  families of vector fields,
\\ their orbits and  the Poincar\'e inequality
}
\author{Annamaria Montanari and Daniele Morbidelli}

\date{\today}

\maketitle

\begin{abstract}\small
   We consider
 a family $\H:= \{X_1, \dots, X_m\}$ of  vector fields in
$\R^n$. Under a suitable \emph{$s$-involutivity} assumption on commutators
of order  at most $ s$,
we show a ball-box theorem for Carnot--Carath\'eodory balls
of the family  $\H$
and we prove  the related Poincar\'e
inequality. Each control ball is contained in a  suitable 
    \emph{Sussmann's orbit} of which we discuss some  regularity properties.
 Our main tool
is a class of \emph{almost exponential maps} which we discuss carefully under low  regularity assumptions
 on the coefficients of the vector fields in $\H$.
\end{abstract}


\section{Introduction and main results}

In this paper we  discuss  Carnot--Carath\'eodory   balls and the Poincar\'e inequality for a family $\H= \{X_1, \dots, X_m\}$ of nonsmooth vector fields in $\R^n$ satisfying 
a suitable
involutivity condition of order $s\in \N$, which turns out to be a good
substitute of the well known H\"ormander's rank hypothesis.
Under our assumptions,  control   balls are not necessarily open
sets in the
ambient space~$\R^n$, but each of them is contained in   a suitable \emph{orbit}
associated
with the vector fields  of~$\H$. In this setting we will   prove a ball-box theorem and  the related Poincar\'e
inequality for control balls of the family~$\H$.
Our main tool consists of  a class of \emph{almost exponential
maps}  which are discussed below.

In the
setting of H\"ormander's vector fields,
  control balls have  been  studied by Nagel,
Stein and Wainger  \cite{NagelSteinWainger}, who proved the following fact:  
   assume  that the (smooth) vector fields $X_j$ of the family $\H$
 together with their commutators of
order at most $s$ span the whole  space~$\R^n$ at any point. Denote by $\P:=\P_s:=\{Y_1,
\dots, Y_q\}$ the family of such commutators. Then, given  the
Carnot--Carath\'eodory ball
$B_{\textup{cc}}(x_0, r)$ associated with $\H$, there are
commutators $Y_{i_1}, \dots, Y_{i_n}\in\P$ of lengths
$\ell_{i_1}, \dots, \ell_{i_n}\le s$ such that
the \emph{exponential map}
\begin{equation}\label{fizz}
\Phi(u):= \exp\Bigl(\sum_{1\le k\le n} u_k r^{\ell_{i_k}} Y_{i_k}\Bigr)x_0
\end{equation}
satisfies a ``ball-box'' double inclusion
$
 \Phi(B_\Eucl(  C^{-1}))\subseteq B_\textup{cc}(x_0, r)\subseteq  \Phi(B_\Eucl(
C))$
where $B_\Eucl(C):= B_\Eucl(0,C)\subset\R^n$ denotes the Euclidean ball of
radius $C>0$ centered at the origin. Moreover, they showed that the Lebesgue measure of control balls is doubling.

More recently, Tao and Wright \cite{TaoWright03} discovered  that maps
$\Phi$ could be manipulated   without  the
Campbell--Baker--Hausdorff--Dynkin formula, using  arguments more  based on
Gronwall's inequality.
Subsequently, Street \cite{Street} extended such approach showing that the H\"ormander's condition can be removed, provided that one assumes that for some $s\in\N$
the following   $s$-integrability condition holds:  for all $Y_i, Y_j\in\P=\P_s$, one can write 
\begin{equation}\label{intello}
 [Y_i, Y_j] = \sum_{1\le k\le q} c_{ij}^k Y_k,
\end{equation}
where the functions $c_{ij}^k$ must have suitable
regularity.  
This condition goes back to Hermann, \cite{Hermann} and
it  ensures that any  Sussmann's orbit $\O_\cal{P}$ of the
family $\cal{P}$
is an integral manifold of the distribution generated by $\P$. 
Under  \eqref{intello}, 
    control balls are contained in
the orbits of the family $\P$  and Street \cite{Street} has shown a complete generalization of
the ball-box inclusion to such setting together with  the doubling estimate for
the pertinent measure of the control ball.

Given a family $\H$ and its \emph{Carnot--Carath\'eodory} distance $d_{\textup{cc}}$, a  remarkable
estimate which embodies many properties of the metric space $(\R^n,
d_{\textup{cc}})$ is the associated Poincar\'e inequality.
It is well known that such inequality
plays a crucial role in several questions
concerning analysis and  geometry, consult  the
references
\cite{FranchiLanconelli83,Jerison,SaloffCoste92,GarofaloNhieu96,Cheeger
,HajlaszKoskela,KeithZhong}, to see the Poincar\'e inequality in action.

The Poincar\'e inequality for H\"ormander vector fields was proved first by Jerison 
in \cite{Jerison}. It was observed in \cite{Jerison}, that
  the natural ``exponential maps''  to prove the
Poincar\'e inequality should be factorizable  as compositions of exponentials
of the original vector fields
of $\H$.
 However, in \cite{Jerison}
the Poincar\'e inequality  was achieved  for   H\"ormander vector fields with
different techniques.\footnote{It must be observed  that Jerison's paper also involves a study of some  nontrivial global aspects of the Poincar\'e inequality which we do not discuss here.}
The program implicitly suggested by Jerison was  carried out   in the subsequent
papers \cite{LM,MM04,MM}.
Namely, in \cite{MM},
the present authors showed that, at least for H\"ormander vector fields (even
with quite rough coefficients), a ``ball-box'' double inclusion still holds
 if we change the map $\Phi$ in \eqref{fizz}  with the
\emph{almost
exponential map} 
\begin{equation*}
      E(h):= \expap(h_1 r^{\ell_{i_1}}Y_{i_1})\circ\cdots\circ 
\expap(h_n r^{\ell_{i_n}}Y_{i_n})(x),
\end{equation*}
where  $\expap$ denote 
the approximate exponentials
appearing in \cite{NagelSteinWainger,VaropoulosSaloffCosteCoulhon,Morbidelli98,MM};
see~\cite[Section~2]{MontanariMorbidelli11d} for the  precise definition.

 In this paper,   starting from  some useful first order expansions of~$E$ obtained in 
\cite{MontanariMorbidelli11d} (see Theorem \ref{deduciamo} below)
we discuss  the structure of control balls  for vector fields belonging to a regularity class 
which we call $\B_s$. We say   that a
family $\H=\{X_1, \dots, X_m\}$
 belongs to the   class  $\B_s$ if all $X_j\in\H$ belongs to
$C^{s}$ 
 (this ensures that all commutatros  $Y_j\in\P$ are $  C^1$); moreover,
we require that \eqref{intello} holds for the family $\P$ and
that the functions $c_{ij}^k$  in~\eqref{intello}  are
$C^1  $ smooth with respect to the  differential structure of each orbit; see Definition~\ref{jdover}.

To state our result we need the following notation.  If $\P= \{Y_1, \dots, Y_q\}$ and $x\in\R^n$,
 then $P_x:=\Span \{Y_j(x):1\le j\le q\}$ and $ p_x := \dim P_x.$
Given $r>0$ and
$Y_{i_1}, \dots, Y_{i_p}\in\P$,  let~$\wt Y_{i_k} = r^{\ell_{i_k}}  Y_{i_k}$ 
be the  scaled commutators and put
\begin{equation}\label{expoenne}
 E_{I,x,r}(h):= \expap(h_1  \wt Y_{i_1})\cdots  \expap(h_p
 \wt Y_{i_p})x
\end{equation}
for each $h$ close to $0\in \R^p$ (after passing to $\wt Y_{i_j}$, the variable
$h$ lives at a unit scale).
We also denote by 
 $\sigma^p$  the $p$-dimensional surface measure and by  $B_\textup{cc}$ control balls.
Finally  $B_\r$ denote balls with respect to the distance $\r\ge \dcc$ defined in  \eqref{coscos}.

\begin{theorem}\label{pocaro}
Let $\H$ be a family of $\B_s$ vector fields. Let $\Omega\subset\R^n$ be a
bounded set. Then there is $C>1$ such that the following holds.  Let $x\in
\Omega$ and take a  positive radius $r<C^{-1}$.
Then there is   a family of $p_x =:p$ commutators
$Y_{i_1}, \dots, Y_{i_p}$ such that the map $E:= E_{I,x,r}$ in \eqref{expoenne}
is $C^1$ smooth  on the  unit ball
$B_\Eucl(1)\subset\R^p$ and satisfies
\begin{align}\label{pochetto}
 C^{-1}\le \frac{\abs{ \p_1E(h)\wedge\cdots\wedge\p_pE(h)}}{\abs{\wt
Y_{i_1}(x)\wedge \cdots\wedge
\wt Y_{i_p}(x)}}
& \le C
\quad\text{for all
$h\in B_\Eucl( 1)$, and}
\\ \label{quickbuild}
 E(B_\Eucl(1))& \supseteq B_{\r}(x, C^{-1}r).
\end{align}
Moreover, 
$E_{I,x,r}$  is one-to-one
on $B_\Eucl\left( 1\right)$  and we have the doubling property
\begin{equation}\label{doppietto} 
 \sigma^p(B_{\textup{cc}}(x, 2r))\le C\sigma^p(B_{\textup{cc}}(x,
r))\quad
\text{for all $x\in\Omega $ and $0<r<C^{-1}.$}
\end{equation}
Finally, 
 for any  $C^1$ function $f$, we have
the Poincar\'e inequality
\begin{equation}\label{pppooo} 
 \int_{B_{\textup{cc}}(x,r)}\abs{f(y) - f_{B_{\textup{cc}}(x,r)}} d\s^p(y)\le
C\sum_{j=1}^m\int_{B_{\textup{cc}}(x, r)} \abs{ rX_jf(y)} d\s^p(y).
\end{equation}
\end{theorem}

The constant $C$ in Theorem \ref{pocaro}  turns out to depend on  an  ``admissible constant'' $L_1$
which will be defined precisely in \eqref{lippo}.
Note that  $L_1$  is defined in terms of the  coefficients
$c_{ij}^k$ in \eqref{intello}  but does not involve any positive 
lower bound on the infimum $\nu(\Omega)$ in \eqref{nu}, which is allowed to vanish even on compact sets.
This makes such result  suitable  in the
perspective of multi-parameter distances studied in \cite{Street}.
Observe  that inequalities 
\eqref{pppooo} and \eqref{doppietto} can be proved in more regular settings using arguments from the papers \cite{Jerison}, \cite{Street} and \cite{MM}. However such arguments do not  provide optimal results;
see the discussion in Section \ref{Sttt}.
Finally, since   
in \eqref{quickbuild}   $B_\r$ denotes the control
ball defined by
\emph{all} commutators (with their degrees, see \eqref{coscos}), as a consequence we have   the local inclusion
$B_\cc(x,r)\supseteq B_\O(x, C^{-1}r^s) $, where $B_\O$ is the geodesic ball on the orbit $\O$; see Remark \ref{picchettina}.

Let us mention that under our regularity assumptions, inclusion \eqref{quickbuild} is not completely trivial. Indeed, such inclusion implies 
 in particular  
the following  fact: a subunit   path $\gamma$ of the family $\P=\{Y_1, \dots, Y_q\}$ of   commutators with $\gamma(0)= :x$,  \emph{cannot leave} the Sussmann's orbit $\O_\H^x$ 
 of the \emph{horizontal} family $\H$.\footnote{Recall that given $\H= \{X_1, \dots, X_m\}$ and $x_0\in \R^n$, then the Sussmann's orbit $\O_\H^{x_0}$ is the
set of points in $\R^n$ which are reachable from $x_0$ via a path which is piecewise  an integral curve
of one among the vector  fields of $\H$; see \cite{Sussmann}.}
This statement needs to be checked carefully. See the discussion in Remark \ref{picchettina} and see Lemma \ref{chop}.
In the H\"ormander case, this issue does not appear, because $\O_\H^x= \R^n$, by Chow's Theorem. 
Indeed, in \cite{MontanariMorbidelli11a}, under the H\"ormander assumption, we are able
to prove Theorem~\ref{pocaro} under even lower  regularity assumptions than those of the present paper: namely, 
we assume that higher order  commutators are $C^1$  only along horizontal directions.

A further delicate part of our argument 
is  the proof of the  injectivity of maps $E$. 
Note that the clever argument by  Tao and Wright, \cite{TaoWright03}, \cite{Street}, is peculiar of the 
standard exponential maps $\Phi$ and 
does not extend to our maps $E$.
Since it does not seem that  any direct
 argument can be adopted,  
  we will let to cooperate the maps $E$ and $\Phi$, which,
although
different, have analogous estimates  on Jacobians.
To accomplish this task, we need first to perform an accurate analysis
of the  standard     exponential  maps $\Phi$. In particular we shall improve
Street's ball-box theorem
for maps $\Phi$ related to to families of  vector fields $\{Y_1, \dots, Y_q\}$ satisfying \eqref{intello},
where $Y_j$ and $c_{ij}^k\in C^1$.  This class is  larger than the class originally studied in \cite{Street}; see especially
the proof of Theorem~\ref{pallascatola}-(ii).
Then we show through a lifting argument that the map
$E_{I,x,r}$ is one-to-one as a consequence of the injectivity of  the map
$\Phi_{I,x,r}$.

Before closing this introduction, we mention some more recent papers where
nonsmooth vector fields are discussed. In \cite{SawyerWheeden},  
 diagonal vector fields are discussed deeply.
In the  H\"ormander case, in the model situation
of equiregular
 families of vector fields,
nonsmooth ball-box theorems have been studied by 
see \cite{KarmanovaVodopyanov,Greshnov,Manfredini}. Finally, \cite{BramantiBrandoliniPedroni2} contains a
nonsmooth lifting theorem.

The paper is organized as follows: In Section~\ref{preliminarmente} we give some preliminaries. In Section~\ref{distratto} 
we prove the ball-box theorem for our almost exponential maps~$E$. In Section~\ref{inusuale} we discuss the ball-box theorem for maps~$\Phi$ for vector fields in the class $\B_s$.
In Section \ref{Sttt} we discuss an approach to the problem for more regular vector fields.


\paragraph{Acknowledgments.  } We thank the referee, who  encouraged us to include the  discussion carried out  in Section 5.

\section{Preliminaries}
\label{preliminarmente}
\paragraph{General notation about constants.}
We denote by $C,
C_0, C_1,C_2 \dots$ large absolute constants. We denote instead by $t_0,
\e_0,\e_1 r_0,
\eta_0,\eta_1,\dots$ or $C^{-1}$ small absolute constant. We will specify
carefully along the
paper what the  constants we deal with  depend on, i.e. what ``absolute'' means.

\paragraph{Vector fields, orbits  and
the control distance.}
Consider a
family    of vector fields $\H=\{ X_1, \dots, X_m\}$
 and assume that   $X_j\in C^1(\R^n)$ for all $j$.
Write $X_j = :f_j\cdot\nabla$, where $f_j\colon\R^n\to
\R^n$.
The vector field $X_j$, evaluated at a point $x\in \R^n$, will be denoted
by  $X_{j,x}$ or $ X_j(x)$.
All the vector fields in this paper are always defined on  the whole
space $\R^n$.
Let 
\begin{equation*}
\begin{aligned}
\dcc (x,y) &:= \inf \big\{r>0: \text{ there is $\gamma\in
\Lip((0,1),\R^n)$
with $\g(0)=x, \g(1)= y$}
\\&
 \qquad\qquad  \qquad\text{and $\dot\gamma(t)\in
\big\{\textstyle\sum_{1\le j\le m} c_j rX_{j, \gamma(t)}: \abs{c}\le
1\big\}$
for a.e. $t\in [0,1]$} \big\}.
\end{aligned}
\end{equation*}
As usual, we call    \emph{Carnot--Carath\'eodory}  or
\emph{control} distance the
distance $\dcc$.

Given a fixed $s\ge 1$, denote by
$\cal{P}   := \{ Y_1, \dots, Y_q\} = \{ X_w: 1\le \abs{w} \le s\}$
the family of commutators of length at most $s$.  Let
$  \ell_j\le s$  be the length of $Y_j$ and
write $Y_j=:g_j\cdot\nabla$.
The distance associated  with $\P$ (where each $Y_j$ has degree $\ell_j$)
 will be
denoted by $\r$:
\begin{equation}\label{coscos}
\begin{aligned}
&\text{$\r(x,y)   := \inf \big\{  r\ge 0 :   $ there is $\gamma\in
\Lip((0,1), \R^n)$ such that $\gamma(0)=x$   }
\\&\quad \text{$\gamma(1) = y$ and $\dot\gamma(t) \in {\big\{ \textstyle{
\sum_{j=1}^q}} b_j r^{\ell_j}Y_j(\g(t)): \abs{b}\le 1\big\}$ for a.e. $t\in
[0,1]$}\big\}.
\end{aligned}
\end{equation}
We denote by  $B_\r(x,r)$, $B_{\textup{cc}}(x,r)$ and $B_\Eucl(x,r)$ the balls of center $x$ and radius $r$ with respect to $\r$, $\dcc$ and the Euclidean distance
respectively. We also denote for brevity $B_\Eucl(r):=B_\Eucl(0, r).$

\begin{definition}[Vector fields of class $\cal B_s$]
\label{jdover}   Let $\H = \{ X_1, \dots,
X_m\}$ be  vector fields in $\R^n$. We say that $\H$ is a family of class
$\cal{B}_s$ if   $X_j\in C_{\Eucl}^{s}$
for $j\in\{1,\dots,m\}$ and
moreover,   given any open bounded set $\Omega_0\subset\R^n$, there is
$C_1>0$ such that we may write for suitable functions $c_{ij}^k  $
\begin{align}\label{int}
[Y_i, Y_j]&  := (Y_i g_j - Y_j g_i)\cdot\nabla = \sum_{1\le k\le q}  c_{ij}^k
Y_k \quad
\text{where }
\\
 \label{int2}  \sup_{x\in \Omega_0}\abs{c_{ij}^k(x)} & \le
C_1\quad\text{for
all $i,j,k\in
\{1,\dots,q\}$;}
\end{align}
we require finally  that for all $i,j,k\in\{1,\dots,q\}$, $\mu\le n$, $x\in
\R^n$ and $I=
(i_1, \dots, i_\mu)\in \{1,\dots,q\}^\mu$, the map
\begin{equation}\label{giarango}
\Omega_{I, x}\ni (u_1, \dots ,  u_\mu)\mapsto
c_{ij}^k\Bigl(\exp\Bigl(\sum_{1\le \a\le \mu} u_\a Y_{i_\a}\Bigr)x\Bigr)
\end{equation}
is  $C^1_\Eucl$ smooth on the open set $\Omega_{I, x}\subset\R^\mu$ where it
is defined.
\end{definition}

\begin{remark}\label{valnaso}  
Class $\B_s$ is a subclass   of the class  $\A_s$  introduced in \cite{MontanariMorbidelli11d}.
More precisely, if a family $\H$ belongs to $\B_s$, then it belongs to $\A_s$ and the constants $L_0$ and $C_0$ in
\cite{MontanariMorbidelli11d} can be estimated by  $L_1 $ in~\eqref{int2}.
\end{remark}

\begin{remark}\begin{enumerate}[noitemsep,label=(\roman*)]
\item The
assumption $X_j\in C^{s}_\Eucl$ ensures that all the vector fields $Y_j$ are
$C^1_\Eucl$ smooth.
It is known  that if \eqref{int} and \eqref{int2} hold with  
$c_{ij}^k$   locally bounded,
then
any \emph{subunit orbit} 
\begin{equation}\label{sobonzo} 
\O^{x_0}_{\P,\textup{cc}}:=\{y\in\R^n: \dcc(x,y)<\infty\}                                                                                  \end{equation}  with topology $\t_{\dcc}$ is an immersed   $C^2$ submanifold and it is an integral
manifold of the distribution 
generated by  $\P$. Charts are described  in \eqref{lagiara}.  In the paper \cite{MontanariMorbidelli11b} we show a more general statement involving Lipschitz vector fields.

\item
Hypothesis \eqref{giarango}  leaves on the orbits $\O = \O_\cal{P}$ of the
family $\cal{P}= \{Y_1,
\dots, Y_q\}$ and it is ensured for instance by the assumption that
$c_{ij}^k\in
C^1_\O  $, i.e.~$C^1$ with respect  to the differential structure of each
orbit. 
\item

Observe also that  conditions \eqref{int} and \eqref{int2} scale
correctely.
Indeed, take a   family  $\H$  of $\cal{B}_s$ vector fields,   denote  $\wt Y_k := r^{\ell_k}Y_k$ for $k=1,\dots, q$ and  $r\in\left]0,1\right]$. Then there are    new $C^1$
functions  $\wh c_{jk}^i(x)$ and an algebraic constant $\wh C_1>0$ so that
$\abs{\wt Y_h \wh c_{jk}^i}\le C_1\abs{\wt Y_h c_{jk}^i} $ for all $i,j,k,h$ 
 and moreover for all $x\in\Omega_0$ we have
\begin{equation}
 \label{stop??}
\begin{aligned}
[\wt Y_j, \wt Y_k]& := [r^{\ell_j}Y_j, r^{\ell_k}  Y_{k}]= \sum_{i=1}^q
\wh c_{jk}^{i}
 \wt Y_i\quad\text{and}\quad
\abs{\wh c_{jk}^i}\le C_1+ \wh C_1.
\end{aligned}
\end{equation}
To see  \eqref{stop??}, 
if $\ell_j+\ell_k>s$, then let $\wh c_{jk}^{i}(x):=
r^{\ell_j+\ell_k-\ell_i}c_{jk}^i(x)$ and we are done.
If instead $\ell_j+\ell_k\leq s$, then the Jacobi identity
  shows that
there are algebraic  constants $a_{jk}^i$ such that $[Y_{j}, Y_k] =
\sum_{\ell_i=\ell_j+\ell_k} a_{jk}^i Y_i$.  Therefore \eqref{stop??} holds.

\end{enumerate}
\end{remark}

Given a family of $\cal {B}_s$ vector fields in $\R^n$ and
$\Omega\Subset\Omega_0\subset\R^n$ bounded sets, introduce
the constant
\begin{equation}\label{lippo}
L_1  : =\sum_{j=1}^m \sum_{ 0\le \abs{\a}\le s
                           }\sup_{\Omega_0}\abs{D^\a
f_j} + \sum_{i,j,k,\ell=1 }^q \big(\sup_{ \Omega_0 }
 \abs{c_{ij}^\ell }+ \sup_{ \Omega_0 }\abs{ Y_k c_{ij}^\ell} \big).
\end{equation}
In the remaining part of the paper we fix open bounded sets $\Omega\Subset\Omega_0\Subset\R^n$ and we consider points $x\in \Omega$ and radii $r\le r_0$ where $r_0$ is small enough to ensure that all balls $B_\r(x,r_0)$ are contained in $\Omega_0$ and that all points $E_{I,x,r}(h)$ and $\Phi_{I,x,r}(u)$  appearing in the paper  belong to $\Omega_0$.

\paragraph{Wedge products and $\eta$-maximality conditions.}
Next,  following \cite{Street}, we define some algebraic quantities  which we will   use below.
 Define for any $p,\mu\in \N$, with $1\le p\le \mu$,
$\I(p,\mu)  := \{I=(i_1, \dots, i_p): 1\le i_1<i_2< \cdots<i_p\le \mu\}  $.
  For each $x\in\R^n$ define
 $ p_x:= \dim  \Span   \{ Y_{j,x} : 1\le j\le q\}.$
Obviousely, $p_x\le \min\{n,q\}$. Then for any $p\in \{1,\dots,\min\{n,q\}\}$,
let
\begin{equation*}
 Y_{I,x} : = Y_{i_1,x}\wedge\cdots\wedge Y_{i_p,x}\in
{\textstyle\bigwedge}_p
T_x\R^n\sim {\textstyle\bigwedge}_p\R^n \quad\text{for all $I\in \I(p,q)$,}
\end{equation*}
and, for all $K\in \I(p,n)$ and $ I\in
\I(p,q)$
\begin{equation}\label{girocollo} 
\begin{aligned}
 Y_I^K(x)& : = dx^K(Y_{i_1}, \dots, Y_{i_p}) (x)
: = \det(g_{i_\a}^{k_\b})_{\a,\b=1,\dots, p}.
  \end{aligned}
\end{equation}
Here we let $dx^K:=dx^{k_1}\wedge \cdots \wedge d x^{k_p}$  for any
$K=(k_1,\dots, k_p)\in
\I(p,n)$.

The family $e_K:= e_{k_1}\wedge\cdots\wedge e_{k_p}$, where  $K\in\I(p,n)$,
gives an othonormal basis of $\bigwedge_p\R^n$, i.e. $\langle e_K, e_H\rangle
= \delta_{K,H}$ for all $K,H$. Then we  have the orthogonal decomposition
$
 Y_I(x)  =\sum_{K}Y_J^K(x) e_K\in {\bigwedge}_p \R^n
$, so that the number
\( |Y_I(x)| : =\bigl(\sum_{K\in \I(p, n)}Y_I^K(x)^2\bigr)^{1/2}  =
\abs{Y_{i_1}(x)\wedge\cdots\wedge Y_{i_p}(x)}
\)
gives the  $p$-dimensional volume  of the parallelepiped  generated by
$Y_{i_1}(x), \dots, Y_{i_p}(x)$.

Let  $I= (i_1, \dots, i_p)\in \I(p, q)$ such that $\abs{Y_I}\neq 0$. Consider
the linear system $\sum_{k=1}^p\xi^k
Y_{i_k}= W$, for some
$W\in\Span\{Y_{i_1}, \dots, Y_{i_p}\}$. The Cramer's
rule  gives the unique solution
\begin{equation}
\label{cromo} \xi^k = \frac{\langle Y_I, \iota^k(W)
Y_I\rangle}{\abs{Y_I}^2}\quad\text{for each $k=1,\dots, p$,}
\end{equation}
where we let $\iota^k_W Y_I:=
\iota^k(W)
Y_I:= Y_{(i_1,\dots, i_{k-1})}\wedge
 W\wedge Y_{(i_{k+1},\dots,i_p)}.$

Let    $r>0$.  Given $J\in \I(p,q)$, let $\ell(J):=
\ell_{j_1}+ \cdots + \ell_{j_p}$. Introduce the vector-valued function
\begin{equation*}
\begin{aligned}
\Lambda_p(x,r)& := \bigl(  Y_J^K (x)r^{\ell(J)} \bigr)_{{J\in \I(p, q)}, K\in
\I(p, n)}
=: \bigl( \wt Y_J^K (x) \bigr)_{{J\in \I(p, q),\,  K\in\I(p, n)}},
\end{aligned}
\end{equation*}
where we adopt the tilde notation $\wt Y_k : = r^{\ell_k}Y_k $ and its obvious
generalization for wedge products. Note that $\abs{\Lambda_p(x,r)}^2 =
\sum_{I\in \I(p, q)}
r^{2\ell(I)}\abs{Y_I(x)}^2$.

Finally, for each $A\subset\R^n$, put
\begin{equation}
\label{nu}
 \nu(A) : = \inf_{x\in  A }  \abs{
\Lambda_{p_x}(x,1)}. \end{equation}

\begin{definition}
 [$\eta$-maximality] Let $x\in\R^n$,  let $I\in \I(p_x,q)$ and $\eta\in (0,1)$.
We say
that $(I,x,r)$ is
$\eta$-maximal if
$\abs{Y_I (x)} r^{\ell (I)} >\eta \displaystyle \max_{J\in \I(p_x, q)} \abs{Y_J(x)}r^{\ell(J)}.
$
\end{definition}
Note that, if $(I,x,r)$ is a candidate to be  $\eta$-maximal with
 $I\in \I(p,q) $, then by definition it \emph{must} be    
$p=p_x$.


\paragraph{Approximate exponentials of commutators.} 
\label{capuozzo}
Let $w_1, \dots, w_\ell\in \{1,\dots,m\}$.
Given $t\in\R$, close to $0$,  define  the approximate exponential 
$\eap^{tX_{w_1 w_2\dots w_\ell}} :=  \expap(t X_{{w_1 w_2\dots w_\ell}})$ as 
in\cite{MontanariMorbidelli11d}   and see also
  \cite{NagelSteinWainger,MM}.
By standard ODE theory,    there is $t_0$  depending on $\ell,  \Omega$,
$\Omega_0$,  $ \sup\abs{ f_j } $ and  $\sup\abs{\nabla f_j}  $     such that
$\exp_*(t X_{{w_1 w_2\dots w_\ell}})x$ is well defined for any $x\in
 \Omega$ and $|t|\le t_0$.
   Define, given $I= (i_1,\dots,i_p)\in\{1,\dots,q\}^p $,   $x\in \Omega$ and
$h\in\R^p$, with
$|h|\le C^{-1}$
     \begin{equation}
\begin{aligned} \label{hhh}E_{I,x}(h)& :=\expap(h_1 Y_{i_1})\cdots
\expap(h_p
Y_{i_p})(x)
\\
\bigl\|h\bigr\|_I & : =\max_{j=1,\dots,p}|h_j|^{1/\ell_{i_j} }\qquad   Q_I(r):
=\{h\in\R^p:\norm{h}_I < r\}.
\end{aligned}\end{equation}

Recall the following result.

\begin{theorem}[{\cite[Theorem~3.11]{MontanariMorbidelli11d}}]\label{deduciamo}
Let $\H$ be a $\B_s$ family.  Let $x\in \Omega$ and let $r\in (0,r_0)$.
Fix $p\in\{1,\dots,q\}$ and $I\in \I(p, q)$. Then the function $E_{I,x,r}$ is $C^1$ smooth on
$B_\Eucl(C^{-1})$.
Moreover, for all $h\in B_\Eucl(C^{-1})$ and for any $k\in\{1,\dots,p\}$
we have $E_*(\p_{h_k})\in P_{E(h)}$ and we can write
\begin{equation}\label{nhb}
E_*(\p_{h_k}) = \wt U_{k,E(h)}  +\sum_{\ell_j=d_k+1}^{s} a^j_k(h)
\wt Y_{j, E(h)}
+ \sum_{i=1}^q  \omega_k^i(x,h)\wt Y_{i,E(h)},
\end{equation}
where, for some $C>1$ we have
\begin{align}\label{sogliola}
       \abs{a_k^j (h)}& \le C \bigl\|h \bigr\|_I^{\ell_j-d_k}\quad\text{for all
$h\in B_\Eucl(C^{-1})$}
\\ \label{merluzzo}  \abs{\omega_i(x,h)} &\le C \bigl\|h \bigr\|_I^{s+1-  d_k}
\quad\text{for all $h\in B_\Eucl(C^{-1})\quad x\in\Omega$}.
\end{align}
\end{theorem}
The proof of Thoerem \ref{deduciamo} in  
\cite{MontanariMorbidelli11d} involves the more general class $\A_s$ and constants in that paper depend on the data $C_0$ and $L_0$ there. Therefore, in view of  Remark \ref{valnaso}, constants in Theorem \ref{deduciamo} depend quantitatively on $L_1$ in \eqref{lippo}.


\paragraph{Gronwall's inequality.}
We shall refer several times to the
following standard fact:
for all $a\ge 0$, $b>0$,
$T>0$ and $f$ continuous on $[0,T]$,
\begin{equation}\label{grroo} 
 0\le f(t)\le a t+ b\int_0^t f(\t) d\t \quad\text{on
$t\in [0, T]$}\quad\Rightarrow
\quad f(t)\le \frac{a}{b} (e^{b t}-1)\quad\text{on  $t\in[0,T]$.}
\end{equation}

\section{Ball-box theorem for almost exponential maps}
\label{distratto}   In this section we prove the ball-box theorem for our almost
exponential  maps $E$ associated with a family  $\H = \{X_1, \dots,
X_m\}$  of  vector fields of class $\B_s$.  
Given   $I\in\I(p,q)$ and $r>0$, we denote as usual  $\wt U_j:  = r^{d_j}U_j : =
r^{\ell_{i_j}}Y_{i_j}$
and
$ E_{I,x,r}(h) : = E(h)= \eap^{h_1 \wt U_1}\cdots
 \eap^{h_p \wt U_p} x$.
Moreover, $Q_I(r)$ denotes the associated box (see \eqref{hhh}). Finally recall
that we use the notation $\abs{\chi}$ to denote the operator norm of any matrix
$\chi$ with real elements.

\begin{theorem}\label{ho1}
Let $\H$ be a family of $\B_s$ vector fields. If $(I,x,r)$ is $\frac
12$-maximal, $x\in\Omega$ and $r< r_0$,
then there are  $C_2>1$ and $
\e_0<1$ such that, for all $j=1,\dots, p:= p_x$,
\begin{equation}\label{bonfo}
   E_*(\p_{h_j})= \wt U_{j,E(h)} + \sum_{1\le k\le p} \chi_j^k(h)
\wt U_{k,E(h)}
\quad\text{for all $h\in Q_I (\e_0)$,}
\end{equation}
where $\chi\in C^0_\Eucl(Q_I(\e_0) , \R^{p\times p})$ satisfies
\begin{equation}\label{dichi}
 |\chi (h)|\le C_2\norm{h}_I\qquad \text{if  $\norm{h}_I\le \e_0 $.}
\end{equation}
\end{theorem}

Note that Theorem \ref{ho1} \emph{does not require} any positive lower bound on the number
$\nu$
defined in~\eqref{nu}. Moreover it holds for the more general class $\A_s$ in 
\cite{MontanariMorbidelli11d} and the  constants $C_2$,   $\e_0$ and
$r_0$
depend quantitatively on $L_0$ and $C_0$ introduced in that paper, and ultimately---in view ol Remark~\ref{valnaso}---on $L_1$.

For future reference, we write \eqref{bonfo} in
 matricial form as
\begin{equation}\label{jojo}
 dE(h)= [\wt Y_{i_1, E(h)},\dots,  \wt Y_{i_p, E(h)}] \,[I_p+\chi(h)]=:[\wt
Y_{I, E(h)}]\,[I_p+\chi(h)].
\end{equation}

\begin{proof}[Proof of Theorem \ref{ho1}]
It follows immediately from Theorem \ref{deduciamo}. Indeed, starting
from \eqref{nhb} and applying  
\cite[Remark~3.3]{MontanariMorbidelli11d} with $\eta= \frac{1}{2}$, we get

\begin{equation*}
\begin{aligned}
 E_*(\p_{h_k}) & = \wt U_{k, E}+ \sum_{\ell_k+1\le \ell_j\le s
}a_k^j(h)\sum_{\a=1}^p b_j^\a \wt U_{\a,E}+\sum_{1\le i\le q}\omega_k^i(x,h)
\sum_{1\le \a\le p}b_i^\a \wt U_{\a,E}
\\& = : \wt U_{k, E} + \sum_{1\le j\le p}\chi_k^j(h) \wt U_{j,E},
\end{aligned}
\end{equation*}
where $E= E(h)$ and, by \cite[Remark~3.3]{MontanariMorbidelli11d}, we have $\abs{b_i^\alpha}\le C$.
 The coefficients $\chi_j^k$ are unique by
the linear independence
of the $\wt
U_{j,E}$. Moreover, since in Theorem   \cite[Theorem~3.11]{MontanariMorbidelli11d}
 we have proved that
$h\mapsto E_*(\p_{h_k})$ is continuous and by assumption $\B_s$ we
know that  the maps $h\mapsto U_j(E(h))$ are continuous, then, the Cramer's rule
\eqref{cromo} shows that $\chi$ is continuous. Finally
estimate \eqref{dichi} follows from the inequality $\abs{a_k^j(h)} + \abs{\omega_k^i(x,h)}\le
C\norm{h}$; see \eqref{sogliola} and \eqref{merluzzo}.
\end{proof}

Next we  discuss  our   ball-box theorem
 in the   class
$\B_s$.

\begin{theorem}\label{punturina}
Let $\H$ be a family of $\B_s$ vector fields.
Then there are $\e_0,\e_1>0$ and $C_2>0$   such that  
for any
  $\frac 12$-maximal triple $(I,x,r)$ with $x\in \Omega$, $I\in \I(p_x,q)$ and
$r\in (0, r_0)$
\begin{itemize*} \item[(a)] for any 
   $\e\leq \e_0$ we have 
\begin{equation}\label{moma}
 E_{I,x,r}(Q_I(\e))\supset  B_\rho (x, C_2^{-1} \e^s r);
\end{equation}
\item [(b)]
the map  $E_{I, x, r}$ is one-to-one on the set
$Q_I(\e_1)$.
\end{itemize*}

\end{theorem}
The proof of inclusion \eqref{moma} will be shown in Lemma \ref{liftato}.  The proof of the injectivity statement will be given later, after some more
work. See page~\pageref{superpuntura}. 
Note that in Theorem~\ref{punturina}, all   constants $\e_0, C_2, r_0$ only  depend quantitatively
on $L_1$ in
\eqref{lippo} and there are no problems even if  the infimum
$\nu(\Omega)$ in
\eqref{nu} is zero.

\begin{remark}\label{picchettina}  Concerning Theorem   \ref{punturina}-(a) note the following aspects. 

\begin{itemize*}
\item[(i)] Inclusion \eqref{moma}   ensures that $B_\r(x, r)\subset \O_\H$. 
We have shown  in   \cite{MontanariMorbidelli11d} that on the orbit    $\O_\H$
there is a topology $\t(\U)$ with basis $\U$ (see \eqref{gasper}), such that
$(\O_\H, \t(\U))$ is a $C^1$  submanifold  and 
 $T_{x}\O_\H= P_x$ for all $x$.\footnote{
Note that even if the vector fields are smooth, maps of the form $E_{I,x}$ are generically
 not much regular. For example, given the smooth vector fields $X_1 = \p_1$
 and $X_2 = (x_1 + x_1^2)\p_2$, then the map 
$
      h\mapsto \expap(h[X_{1}, X_2](0,0) =(0, h+h\abs{h}^{1/2})
$
is $C^{1,1/2}$ only; see  \cite[Example~5.7]{MM}.}
 Therefore, to show inclusion \eqref{moma} one must first  give a rigorous proof of the fact that  
a subunit path of the family $\P$ starting from $x\in\O_\H$ should stay in $\O_\H$
for  $t$ close to zero. We prove this   statement   in Lemma \ref{chop} where we show that
  the  subunit orbit~$\O_{\P,\textup{cc}}$
 of the commutators (see \eqref{sobonzo}) coincides with
the Sussmann's  orbit~$\O_\H$  of the original vector fields.

\item[(ii)] Note also that  \eqref{moma}
implies the Fefferman--Phong-type local  inclusion
 $B_d(x, r)\supseteq B_{\O}(x, C^{-1}r^s),$
where $B_\O$
denotes the geodesic ball on $\O$. Here $x$ belongs to a compact set and $r$ is
small enough: see \cite{FeffermanPhong81}. 
Therefore   the
  topology $\t_{\dcc} $   on $\O:=\O_\H = \O_{\P, \cc}$ is equivalent to the
topology defined by the metric $\r$ and both are equivalent
to the topology associated with the geodesic Riemannian  distance provided by
the first fundamental form of $\O$. 
\end{itemize*} \end{remark}

The main application of the results in this section is the following. 
\begin{corollary}\label{poincuss}
   Let $\H = \{X_1, \dots, X_m\}$ be a family of $\B_s$ vector fields in $\R^n$.
Then for any bounded open set $\Omega\subset\R^n$,  there is $C>1$
depending on
$L_1$ in \eqref{lippo} such that for any  $x\in \Omega$ and $r\in
(0, C^{-1}]$, letting $p:= p_x$, we have
\begin{align*}
 \sigma^p(B_{\textup{cc}}(x, 2r)) & \le C\sigma^p(B_{\textup{cc}}(x, r))\quad\text{and}
\\
 \int_{B_{\textup{cc}}(x,r)}\abs{f(y) - f_{B_{\textup{cc}}(x,r)}} d\s^p(y)
& \le
C\sum_{j=1}^m\int_{B_{\textup{cc}}(x, r)} \abs{ rX_jf(y)} d\s^p(y).
\end{align*}
\end{corollary}

The doubling property was already proved by
Street \cite{Street} under   more restrictive assumptions. At the
author's knowledge, the Poincar\'e inequality in such  setting, is
new even in the smooth case.

\begin{proof}[Proof of Corollary \ref{poincuss}]The proof of the doubling
property is an immediate consequence of
Theorems \ref{ho1},  \ref{punturina} and
of area formula.
The proof of the Poincar\'e inequality  can be obtained arguing
as in \cite{LM}.
We avoid here the repetition of the arguments.
\end{proof}

Before starting the proof of Theorem \ref{punturina}-(a),  recall that  it was shown in \cite[Theorem~3.13]{MontanariMorbidelli11d} that maps of the form $E_{I,x} $ can be used to give to  $\O_\H$ a structure of $p$- dimensional integral manifold of the distribution generated by $\P$. More precisely, one can introduce a topology
 $\t(\cal{U})$ generated by the family
 \begin{equation}\label{gasper} 
\begin{aligned} 
\U &:=\{ E_{I, x}(O):  x\in \O, I\in
\I(p,q),\abs{Y_I(x)}\neq 0   
\\ &\qquad \qquad \text{and  $O\subset O_{I, x}$ is a  open
neighborhood
of the origin$\}.$}
\end{aligned}             \end{equation}  
(here $O_{I,x}$ is a neighborhood of the origin such that  $E_{I,x}(O_{I,x})$ is an 
embedded submanifold) and maps $E_{I,x}$ can be used as charts.

In order to prove  \eqref{moma}, we need the following  lemma. Let $\r$ be the distance with respect to the family $\P$ defined in  
\eqref{coscos}. Let $\O_{\P, \cc}^x:=\{y\in\R^n: \r(x,y)<\infty \}$ be the subunit orbit of the family $\P$ (see \eqref{sobonzo}) and let $\t_\r$
 be the topology associated with~$\r$.

\begin{lemma}\label{chop} 
    Let $\H$ be a
 family in $\B_s$ for some $s$. Let ${x_0}\in\R^n$. Then we have the following topologically 
continuous inclusions:
\begin{equation*}
(\O_\H^{x_0}, \t(\U))\stackrel{\mathrm{(a)}}{\subseteq}      
\bigl(\O^{x_0}_{\P,\cc}, \t_{\r} \bigr)
\stackrel{\mathrm{(b)}}{\subseteq}(\O_\H^{x_0}, \t(\U)).
\end{equation*}
\end{lemma}
\begin{remark}
Note that  on $\O^{x_0} := \O_{\P,\cc}^{x_0}=\O_\H^{x_0}$ both inclusions $(\O, \t(\U))\subseteq(\O, \t_\cc)\subseteq
(\O, \t_\r)$ are trivially continuous. Therefore,  Lemma \ref{chop} shows that all mentioned topologies are equivalent on  $\O^{x_0}$.
\end{remark}

The proof of Lemma \ref{chop} relies on the following facts discussed in \cite{MontanariMorbidelli11b}. Let $\P=\{Y_1, \dots,Y_q\}$
 be a family of $C^1$ vector fields satisfying \eqref{int} and \eqref{int2}.
Fix a subunit  orbit $\O_{\P,\cc}$. Then $p_x=: p$ is constant as $x\in \O_{\P,\cc}$ and moreover $(\O_{\P,\cc},\t_\r)$ is a 
$C^2$ integral manifold of the distribution spanned by $\P$. See \cite{MontanariMorbidelli11b}.
Charts can be described as follows. For any $x\in\O_{\P,\cc}$ and for each $I\in\I(p,q)$ such that $\abs{Y_I(x)}\neq 0$ there are $\e$, $\delta>0$ and $\b\in 
C^1(B_\Eucl(x,\e),\R^{p\times p})$ such that the vector fields
$      V_j:= \sum_{k=1}^p \b_j^k Y_{i_k}$, where $j=1,\dots, p$, are
$C^1$ smooth of $B_{\Eucl}(x,\e)$ and 
satisfy $[V_j, V_k](\xi)=0$ for all $\xi\in B_\rho(x,\delta)\subset B_\Eucl(x,\e)$ where $\r$ is
defined in \eqref{coscos}. Moreover, the map
\begin{equation}\label{lagiara} 
      \Psi_{I,x}(u):= \exp\Big(\sum_{1\le j\le p}u_j V_j\Big)x
\end{equation}  
is a $C^2$ full rank map from a neighborhood $O_{I,x}$ of the  origin which parametrizes a $C^2$ embedded submanifold $\Psi_{I,x}(O_{I,x})$ which satisfies  $T_{\psi_{I,x}(h)}\Psi_{I,x}(O_{I,x})= P_{\Psi_{I,x}(h)}$
for all $h\in O_{I,x}$. Furthermore, the family 
$\cal{S}:=\{\Psi_{I,x}(O): O\subset O_{I,x}$  is an open neighborhood of the origin$\}$ can be used as a base for a topology $\t(\cal{S})$ on $\O_{\P, \cc}$ which is equivalent to $\t_{\r}$.

All these facts have been proved in \cite{MontanariMorbidelli11b} for Lipschitz vector fields
and in particular hold in our case.

\begin{proof}[Proof of Lemma \ref{chop}]
      Inclusion
 (a) is obvious together with its continuity. Indeed, we always have
 $B_\r(x,r)\supset E_{I,x}(\{\|h\|_I< C^{-1}r\})$ for all $x,r$ and  for
 some universal $C$.

To prove (b), we use the topology $\t(\cal{S})$ instead of $\t_\r$. Let 
 $\Sigma$ be a  $\t(\U)$-neighborhood of some fixed $x\in \O_\H$.
 Taking $I\in\I(p,q)$ such that $\abs{Y_I(x)}\neq 0$, we may assume that for some neighborhood $O$
of the origin 
 $\Sigma \supset E_{I, x}(O)$, where $E_{I, x}(O)$ is a $C^1$ embedded $p$-dimensional 
submanifold. Possibly taking a smaller  $O$, we may assume that $E_{I,x}(O))\cap B_\Eucl(x,\delta)$ is a $C^1$ graph. 
We claim that there is $\sigma >0$ such that 
the inclusion $\Psi_{I,x}(B_\Eucl( \sigma))\subset E_{I, x}(O)$ holds. 
This will conclude the proof. To show this claim,  note that, given $u\in B_\Eucl(\sigma)$, we can write 
$\Psi_{I, x}(u)= \gamma(1)$, where $\gamma$ is the integral curve of the $C^1$
 vector field $\sum_j u_j V_j$. Since the vector fields $V_j$ are $C^1$, 
 the required statement follows if $\sigma $ is small enough by an application of Bony's 
theorem \cite[Theorem~2.1]{Bony69}. 
\footnote{Recall that an aplication of  Bony's theorem states that, if $\Sigma\subset \R^n $ with a topology $\t$
is a $C^1$ immersed submanifold of $\R^n$ and $V$ is a locally Lipschitz vector field such that 
$V(x)\in T_x\Sigma$ for all $x\in \Sigma$, then for all $x\in \Sigma$, $e^{tV}x\in \Sigma$ for $t$ close to $0$. More precisely, for all $\Omega\in\tau$ and $x\in\Omega$ there is $t_0$ such that $e^{tV}x\in \Omega$ if $\abs t\le t_0.$}
\end{proof}

An alternative proof of  (b) relies on the fact that if $\abs{Y_I(x)}\neq 0$, then for all $O\subset O_{I,x}$
the map $E_{I,x}|_O$ with values into the $C^2$ manifold $  \O_{\P, \cc}$ is   $C^1$ and
nonsingular. Therefore it is open, 
because the dimensions of $O$ and $\O_{\P,\cc}$ are the same.

 The following  lifting lemma implies  Theorem \eqref{punturina}-(a).

\begin{lemma}\label{liftato}
Let $\H$ be a family of $\B_s$ vector fields. If $(I,x,r)$ is $\frac
12$-maximal, $x\in\Omega$ and $r< r_0$,
then there are  $C_2>1$ and $
\e_0<1$ such that  
 for all
$\e\leq \e_0$
letting
$C_\e:=C_2\e^{-s} $
 the following holds: let $\gamma$ be a Lipschitz path such
that $\gamma(0)=x$,
$
 \dot\gamma = \sum_{j=1}^qc_j (C_\e^{-1} r)^{\ell_j}Y_j(\gamma)$  a.e. on
$[0,1]$,
where $\abs{c}  \le 1$. Then
 there is a  Lipschitz continuous path $\theta:[0,1]\to \R^n$ such that
$\theta(0)=0$,
$E_{I,x,r}(\theta(t))=\gamma(t)$ 
and $\norm{\theta(t)}_I\le \e$
for all  $t\in[0,1]$.
\end{lemma}

Before giving the proof of the lemma, recall that if  $p\in\N$ and   $\chi, b \in
\R^{p\times p}$,  then
\begin{equation}\label{posteggio}
|\chi|\le \frac 12\quad\Rightarrow \quad
\bigl|(I_p+\chi)^{-1}(I_p+b)-I_p\bigr|\le
2\left(\abs{\chi}+\abs{b}\right)
\quad\text{for all}\quad  b\in \R^{p\times p}.
\end{equation}
This can be seen by writing $(I_p+\chi)^{-1}$ as  a Neumann series.

\begin{proof}[Proof of Lemma \ref{liftato}] 
The argument  of the proof is  analogous to
\cite{NagelSteinWainger,MM}. We include the argument because it will be used in
Proposition \ref{loft}.

First of all, by Lemma \ref{chop}-(b), we know that $\gamma$ belongs to
$ \O_\H$. 
  Let $\e\le \e_0$ and define $C_\e:= C_2\e^{-s}$, where  the
constant $C_2$ will be fixed soon.
Let $\bar t\in[0,1]$. We say that $\theta\in \Lip_\Eucl([0, \bar t], \R^p)$
  is an
$\e$-lifting of $\gamma$ on $[0, \bar t]$ if $\theta(0)= 0$, $E\circ \theta =
\gamma$ on $[0, \bar t]$ and $\norm{\theta(t)}_I\le \e$ for all $t\in
[0, \bar t]$.
Let
$ t_0:= \sup \bigl\{ \bar t\in[0,1]:$ there is a $\e$-lifting
 of $\gamma$ on $[0,\bar t] \bigr\}$.
We already know that $t_0>0$. Our purpose is to show that $t_0=1$.

Next we  claim that if $\theta$
is an $\e$-lifting of $\gamma$ on $[0,\bar t]$, then it should be
\begin{equation}\label{journal}
\norm{\theta(t)}_I\le \frac \e 2 \quad \text{for all $t\in[0, \bar t]$.}
                                          \end{equation}
In order to prove \eqref{journal}, Let $t^*\in (0, \bar t)$. In a neighborhood
$O^*$ of $\theta(t^*)$ the map $E: O^*\to E(O^*)$ is a $C^1$ diffeomorphism onto
an open
neighborhood $E(O^*)$ of $\gamma (t^*)$ in $\O$. Let
$F$ be its inverse. Then for a.e. $t$ close to $t^*$ we get for all
$k\in\{1,\dots,p\}$
\begin{equation*}
\begin{aligned}
\frac{d}{dt}\theta^k(t)=\frac{d}{dt} F^k(\gamma_t)  &=
\sum_{1\le \b\le q} c_\b(t)
C_\e^{-\ell_\b}\wt Y_\b F^k (\g_t)
=
\sum_{1\le \b\le q} c_\b(t) C_\e^{-\ell_\b}
\sum_{1\le j\le p} b_\b^j \wt Y_{i_j} F^k (\g_t) .
 \end{aligned}
\end{equation*}
Here $\wt Y_{i_j}:= r^{\ell_{i_j}}Y_{i_j}$. Differentiating the
identity  $(F\circ E)(h)= h$ for $h\in O^*$, we also get
$ I_p= d(F\circ E) = dF(E)dE =dF(E) [\wt Y_I (E)] (I_p+\chi )$. 
Letting  $I_p+\mu = (I_p+\chi)^{-1}$, we obtain
$\abs{\wt Y_{i_j} F^k} =\abs{  \delta_j^k + \mu_j^k}\le C$
for all $j,k=1,\dots, p$.
  Observe that
$\abs{I_p+\mu}\le C$, by \eqref{posteggio} with $b=0$.
Therefore $\bigl|\frac{d}{dt}\theta^k(t)\bigr|\le C C_\e^{-1}$
for all $t\in [0, \bar t)$.

Now we are in a position to prove estimate \eqref{journal}. Assume that it  is
false. Then, there is $\wt t\in (0, \bar t)$ such that
for all $t\in [0, \wt t)$ we have
$\norm{\theta(t)}<\frac{\e}{2}=\norm{\theta(\wt t)}$.
Therefore, we get for some $k\in \{1, \dots, p\}$,
\begin{equation*}
 \Big(\frac{\e}{2}\Big)^{d_k}= |\theta^k(\wt t)| = \Big|
\int_0^{\wt t} \frac{d}{d\t}\theta^k(\t) d\t \Big|\le
C C_\e^{-s} = C C_2^{-1}\e^s.
\end{equation*}
 Therefore, if $C_2$ is large enough to ensure that  $C
C_2^{-1}<\frac{1}{2^s}$,   this
chain of inequalities  can not hold. This shows \eqref{journal}.

At this point, it is easy to check that an $\e$-lifting on $[0, \bar t]$ is
unique. Indeed, if there were two different liftings $\theta_1 , \theta_2$, then
the set $\{t\in[0, \bar t]: \theta_1(t) = \theta_2(t)\}$ would be nonempty, open
and
closed in $[0, \bar t]$. This implies that $t_0$ is actually a maximum. To
conclude the argument, observe that it can not be $t_0<1$, because in this case
we could extend the lifting on  a small interval $[0, t_0+\delta]$, for some
$\delta>0$. The proof is concluded.
\end{proof}

\begin{remark}
      Note that the constant  $C_2$ depends  quantitatively on the constant  $C_0$ and $L_0$ in
\cite{MontanariMorbidelli11d}. See Remark \ref{valnaso}. In the particular 
 case where $\H$  satisfies the H\"ormander condition at step $s$, then we have  $\O_\H= \O_{\P, 
\cc}= \R^n$ and 
Lemma~\ref{liftato} holds with  $C_2$ depending on $C_0$ and $L_0$.
\end{remark}

We are left with the proof of Theorem \ref{punturina}-(b).
To prove  such statement,
we need a
multidimensional version of the lifting statement just proved and we also
need an \emph{ad hoc} version of Street's  ball-box Theorem  \cite{Street}
(this will be discussed in    Section \ref{inusuale}).

Let $\eta_1$ be the constant in Theorem \ref{pallascatola}.
Fix $\eta_2\le \eta_1$ small enough to ensure that
 \begin{equation}\label{etto}
C_6\eta_2^{1/s}\le C_2^{-1}\e_0^s,\end{equation}
  where
$C_6$ and $ \eta_2$ appear  
in \eqref{strattone}, while $C_2$ and $\e_0$ denote the constants in  the
already proved Theorem \ref{punturina}-(a).  Note that the constant $C_6$ in \eqref{strattone}
is completely independent of the results of the present section.
  Therefore \eqref{moma} and \eqref{etto} 
give the inclusions
\begin{equation*}
 E\big(Q_I(\e_0)\big) 
\supset
B_\rho(x,
C_2^{-1}\e_0^s r)
\supset
 B_\rho(x, C_6\eta_2^{1/s}r)
\supset
\Phi(B_\Eucl(\eta_2))
\supset
B_\rho(x, C_6^{-1}\eta_2^s r),
\end{equation*}
where we kept \eqref{strattone} into account in last inclusion.
Here $(I,x,r)$
is $\eta$-maximal, $E:= E_{I,x,r}$ and  $\Phi:= \Phi_{I,x,r}$.

Here is our lifting result for the  maps $\Phi$.

\begin{proposition}[lifting of standard exponential maps]\label{loft}
Let $\H$ be a  $\B_s$ family. Let $\eta_2$ be
a constant satisfying \eqref{etto}, let $(I,x,r)$ be a
$\frac 12$-maximal triple and let $\Phi= \Phi_{I,x,r}$ and $E:= E_{I,x,r}$ be
the
corresponding  maps. Then there are
$\eta_3\le\eta_2$, $C_3>1$ and
  $\theta\in C^1_\Eucl(B_\Eucl( \eta_3),  Q_I(\e_0)) $ such that
$\theta(0)=0$,
\begin{equation}\label{p1}
E(\theta(u))= \Phi(u)\qquad\text{for all $u\in B_\Eucl( \eta_3) $ }
\end{equation}
and, letting
$d\theta(u)=:I_p + \omega(u)$, we have
\begin{equation}\label{p2}
 \abs{\omega(u)}\le C_3\abs{u}^{1/s}\le \frac 12 \quad\text{for all
$u\in B_\Eucl( \eta_3) $ }.
\end{equation}
The constants $\eta_3$ and $C_3$ depend on   $L_1$
 in \eqref{lippo}.
\end{proposition}

From now on, we restrict  the choice of $\e_0$ in Theorem \ref{ho1} and Lemma
\ref{liftato}  in order to ensure that
\begin{equation}
 \label{rich}
C_2\e_0\le \frac 14,
\end{equation}
where $C_2$ appears in \eqref{dichi}.

Taking for a while Proposition \ref{loft} for granted, we are ready to prove
the injectivity statement of Theorem~\ref{punturina}.

\begin{proof}[Proof of Theorem \ref{punturina}-(b)]\label{superpuntura}
 We combine the just stated proposition  with  Theorem
\ref{pallascatola}.
Let
$\eta_3$ be the constant in Proposition \ref{loft}.
Since  $\eta_3\le \eta_1$, where $\eta_1$ is  the constant in
Theorem \ref{pallascatola}, 
$\Phi$  must be   one-to-one on $B_\Eucl( \eta_3)$. Thus,  $\theta$ is
one-to-one
on the same set  and  $E$ is   one-to-one on
$\theta(B_\Eucl(\eta_3))$.
Clearly, estimate \eqref{p2} implies that
$
\frac 12\abs{u-\wt u}\le\abs{\theta(u)-\theta(\wt u)}\le 2
\abs{u-\wt u}$, for all $u,\wt u\in B_\Eucl( \eta_3) $. 
Therefore, $\theta(B_\Eucl(\eta_3))\supseteq B_\Eucl(\eta_3/2)$. The
proof is concluded taking $\e_1 =\eta_3/2 $. \end{proof}

\begin{proof}[Proof of Proposition \ref{loft}]
The proof is articulated in three steps.

\step{Step 1.}
Take  $\eta_3$ so small that $\eta_3^{1/s}\le
C_2^{-1}\e_0^s$, where $C_2$ is the constant in Lemma~\ref{liftato}.
Then for any $\wt \eta\le \eta_3$ and for any
  $\theta\in C^1_\Eucl(B_\Eucl( \wt \eta), \R^p)$   such that $\theta(0)=0$ and
$E(\theta)=
\Phi$ on $B_\Eucl(\wt\eta)$,
 we have
\begin{equation}
\norm{\theta(u)}_I\le \frac{\e_0}{2}\qquad\text{for all $u\in B_{\Eucl}(\wt
\eta)$.}
\end{equation}

To accomplish Step 1, assume that a lifting $\theta$ enjoying the described
properties is given. Let $u\in B_\Eucl( \wt\eta)$ and look at the path
$\gamma(t)= \Phi(tu)$, where $t\in [0, 1]$. Our choice of constants ensures that
there is a unique lifting
$\lambda\in \Lip [0,1]$, such that $\lambda(0)= 0$ and
$E(\lambda(t)) =
\gamma(t)$ on $[0,1]$. (In fact here $\lambda$ is $C^1$ smooth, because
$\gamma\in C^1$.) Moreover, see estimate \eqref{journal}, we have
$\bigl\|\lambda(1)\bigr\| \le
\frac{\e_0}{2}$. Since by uniqueness it must be $\theta(tu) = \lambda(t)$ for
all $t$,
Step 1 is accomplished.

\step{Step 2.} Let $\wt\eta\le\eta_3$ and let $\theta \in
C^1(B_\Eucl( \wt\eta))$
 such that $\theta(0)=0$ and $E\circ\theta = \Phi$  holds on
$B_\Eucl(\wt\eta)$.
 Then we claim that  \eqref{p2} holds on $B_\Eucl( \wt\eta)$.

To prove the claim, observe that by Step 1 we know that  $\norm{\theta(u)}\le
\frac{\e_0}{2}$ for all $u\in B_\Eucl(\wt\eta)$. Therefore
\eqref{jojo} gives
\begin{equation*}
d\Phi(u)=   dE(\theta(u)) d\theta(u)
= [\wt Y_I(\Phi(u))]\,[I_p+\chi(\theta(u))] d\theta(u).
\end{equation*}
Combining with  \eqref{mutti}, which states that  $
d\Phi(u)=[\wt Y_I(\Phi(u))]\, [I_p +b
(u)]$, we conclude that
$
 d\theta(u)
=[I_p+\chi(\theta(u))]^{-1}[I_p+b(u)]=:I_p+\omega(u).
$
To estimate $|\omega|$ observe that $\norm{\theta(u)}\le \e_0/2$, by Step 1.
Therefore,
\eqref{dichi} gives
$|\chi(\theta(u)|\le C_2 \norm{\theta(u)} \le \frac 12 C_2\e_0\le
\frac 18,$
by  requirement \eqref{rich} on   $\e_0$.
Then \eqref{posteggio} gives
$ \abs{\omega(u)} \le 2(|\chi(\theta(u))|+ |b(u)|)
\le\frac 14 + 2
C_4 \eta_3\le \frac 12,
$
if we choose $\eta_3$ small enough. Here $C_4$ is the constant appearing in
    \eqref{muto}.
 Thus  $\Lip_\Eucl(\theta; B_\Eucl(\wt
\eta))\le 2$ and moreover
\begin{equation*}
\begin{aligned}
 |\omega(u)|&\le 2 \big(|\chi(\theta(u))| +
|b(u)|\big)\le 2\big(C_2 \norm{\theta(u)}+ C_4|u|\big)
\le  C_3\abs{u}^{1/s},
\end{aligned}
\end{equation*}
for some   $C_3>1$ depending on $L_1$ only. Therefore  \eqref{p2} is completely
proved and Step 2 is finished.

\step{Step 3.} Let $\Omega_1,\Omega_2\subset B_\Eucl(\eta_3) $ be connected
open sets. Assume that $\Omega_1\cap\Omega_2$ is connected and that $0\in
\Omega_1$.  Let
also  $\theta_i\in C^1_\Eucl(\Omega_i, \R^p)$ 
 be such that
  $E\circ\theta_i = \Phi$, on $\Omega_i$ for $i=1,2$. Assume finally that
$\theta_1(0)= 0$ and that $\theta_1(u_0) = \theta_2(u_0) $ for some $u_0\in
\Omega_1\cap\Omega_2$.
Then it must be $\theta_1= \theta_2$ on $\Omega_1\cap \Omega_2$.

To prove Step 3, let $A:=\{ u\in \Omega_1\cap\Omega_2 : \theta_1(u)=
\theta_2(u)\}.$ Note that $A\neq\varnothing$ because $u_0\in A$. We
show that
$A$ is open and closed in $\Omega_1\cap\Omega_2$. To see that $A$
is open, let $\wt u \in A$ and let $\wt h= \theta_1(\wt u)= \theta_2(\wt u)$. By
Step 1
we know that $\bigl\|\wt h\bigr\|\le \frac{\e_0}{2}$. Since the map $E$ is
nonsingular,  there is a
neighborhood $\wt O $ of $h$
 such that
$E|_{\wt O}:\wt O\to E(\wt O )\subset\O$
is a $C^1$ diffeomorphism.
Let $\wt F$ be its inverse. Note also that, since
the maps $\theta_i$ are continuous, we may assume that for a small open set $\wt
V$ containing $\wt u$ and contained in
  $  \Omega_1\cap\Omega_2$, we have $\theta_i(\wt V)\subset \wt O$.
Therefore, starting from  identity $E(\theta_1(u)) = E(\theta_2(u))$ for all
$u\in \wt V$, we can apply $F$ and we get $\theta_1(u)= \theta_2(u)$ for all
$u\in \wt V$. This shows that $A$ is open.

Finally, to show that $A$ is closed, let $u_n\in A$ for all $n\in \N$, $u_n\to
u\in \Omega_1\cap\Omega_2$, as $n\to \infty$.   Then, the continuity of
$\theta_1$ and $\theta_2$
ensures
that $\theta_1(u)= \theta_2(u)$, as desired.

\step{Step 4.}
Finally, we show that the lifting exists. Let $\wt \eta: =\sup\{\eta\in (0,
\eta_3]:
$
there is $\theta\in C^1(B_\Eucl( \wt \eta), \R^p)$  such that $\theta(0)=0 $
and
$E(\theta)
=\Phi$  on $B_\Eucl( \wt\eta)\}$.
We will show that $\wt \eta= \eta_3$.

To show Step 4,  assume that $\wt \eta <\eta_3$ strictly. Let $(\eta_n)$ be a
sequence
with $\eta_n\nearrow \wt \eta$. Then, there are $\theta_n\in
C^1_\Eucl(B_\Eucl( \eta_n), \R^p)$ with $\theta_n(0)=0$ and
$E\circ\theta_n = \Phi$ on $B_\Eucl(\wt \eta)$.
By Step 3, there is a unique $\wt \theta\in C^1(B_\Eucl(\wt \eta))$ which
extends all the maps $\theta_n$.
Note that the map  $\wt\theta$ is
$1/2$-biLipschitz up to $\overline{B_\Eucl( \wt\eta)}=: \wt B$, by Step 2.
Now, fix a point $u_1\in \p \wt B $. Let $B_\Eucl(u_1, \d_1)\subset B_\Eucl(\eta_3)$ 
be a ball of
sufficiently small radius $\d_1$ so that $\wt \theta(B(u_1, \d_1)\cap  \wt
B )\subset O$, where $O$ is a neighborhood of $\wt\theta(u_1)$ such that
$E|_{O}:O\to E(O)\subset\O$
is a $C^1$-diffeomorphism (we can equip $\O_\H= \O_{\P, \cc}$ with  the $C^2$ differential structure on
 $\O$ 
described by the family of charts of the form \eqref{lagiara}).   Let $F: E(O)\to O$ be its inverse.
The set $\Phi^{-1}(E(O))$ contains the ball $B_\Eucl(u_1, \delta_1')$ for some
$\delta_1'\le \delta$. 
 We can define the
map
$\theta_1(u):= F(\Phi(u))$ for all $u\in B(u_1, \d_1')$.
Therefore, by Step 3,
 we have extended the lifting to the domain $\wt B\cup B(u_1, \d_1')$. Iterating
a finite number of times we discover that the map $\wt \theta$ can be extended
to a larger ball $B_\Eucl(\wt\eta+\d)$, for some small $\delta>0$. Therefore it
can not be $\wt \eta
<\eta_3$ strictly  and the proof is concluded.
\end{proof}

\section{Ball-box theorem for standard exponential maps}
\label{inusuale}

Here we prove a
ball-box theorem for the  exponential maps $\Phi$ associated with a family $\P= \{Y_1, \dots, Y_q\}$ of vector fields. 
 We use the methods
introduced in \cite{TaoWright03} and \cite{Street}. However,  since we assume
less regularity
than \cite{Street}, we need to modify slightly some of    the original
techniques.

We keep our usual notation. Given a family $\cal{P}=\{Y_1, \dots, Y_q\}$ of $C^1$ vector fields, with degrees  
$\ell_1, \dots, \ell_q\le s$,  we write  $Y_j = g_j\cdot\nabla$.
 Denote by $B_\r$ balls
with respect to the distance   $\r$ defined in \eqref{coscos}.
It is known that if there are locally bounded coefficients $c_{ij}^k$ such that 
\eqref{intello} holds, then any orbit $\O^{x_0}_{\P,\textup{cc}}:=\{y\in\R^n: \dcc(x,y)<\infty\}$ with topology $\t_{\dcc}$ is an immersed   $C^2$ submanifold and it is an integral
manifold of the distribution generated by  $\P$.  (In the paper \cite{MontanariMorbidelli11b} we show a more general statement involving Lipschitz vector fields.)
  Here we assume that $c_{ij}^k$ are $  C^1$-smooth on each orbit $\O$.
Introduce the admissible constant 
\begin{equation}
\label{elledue}
L_2:
=\sum_{i,j,k,\ell=1 }^q \sup_{ \Omega_0 }
 \abs{c_{ij}^\ell }+ \sup_{ \Omega_0 }\abs{ Y_k c_{ij}^\ell}.  \end{equation} 

\begin{remark}
      If a family $\H= \{X_1, \dots, X_m\}$ belongs to $\B_s$, then the constant $L_2$ associated with the
family $\cal{P}= \{Y_1,\dots, Y_q\}$ satisfies obviously $L_2\leq L_1$, see \eqref{lippo}.
\end{remark}

Let $\Omega\Subset\Omega_0$ be the fixed sets introduced after \eqref{lippo}.  
Fixed $x_0\in\Omega$,  $r>0$ and
$I\in \I(p_{x_0}, q)$, define for $u$ close to the origin
\begin{equation}\label{fi}
\Phi(u): =\Phi_{I,x,r}(u):=  \exp\Big(\sum_{1\le j\le p} u^j \wt
Y_{i_j}\Big)(x_0)
\end{equation}
  where, for $k=1, \dots, q$, we let  $\wt Y_k = r^{\ell_k} Y_k =
\sum_{\a=1}^n
\wt g_k^\a \p_\a.$  If $\abs{Y_I(x_0)}\neq 0$ and $\delta>0$ is small
enough,
then the map
$\Phi\big|_{B_\Eucl( \delta)}:B_\Eucl( \delta)\to\Phi(B_\Eucl(
\delta))\subset\O$ is a
$C^1$ diffeomorphism. Here we equip $\O$ with  the $C^2$ differentiable
structure
given by charts of the form \eqref{lagiara}. The inverse map  $\Psi:=
(\Phi|_{B_\Eucl(\d)})^{-1}$ is a $C^1$
chart on $\cal{O}$. Note that a map $f:\O\to \R$ is $C^1_\O$ if $f\circ\Phi$ is
$C^1_\Eucl$ for all charts of such family.

\begin{theorem}
\label{pallascatola}Let $\P $ be a
family of $ C^1$ vector
fields. Assume that there are  functions $c_{ij}^k$
locally bounded in  $\R^n$ such that 
the integrability condition \eqref{intello} holds at any point.  
Assume also that $c_{ij}^k\in C^1(\O)$.
Let  $(I, x_0, r) $ be  
$ \frac{1}{2}$-maximal, where $x_0\in \Omega$, $r\le r_0$ and $I\in
\I(p_{x_0}, q)$.
Let $p_{x_0}= :p$ be the (constant  on $\cal{O}$) dimension of $P_{x_0}$.
Then there are constants  $\eta_1 ,C_6, C_5 >0$ depending on $L_2$ in
\eqref{elledue} such that
\begin{enumerate*}
\item [(i)] 
there is $A\in C^1_\Eucl(B_\Eucl( \eta_1), \R^{p\times p})$ such that the
vector fields
  $
Z_j= \p_{u_j} + \sum_{k=1}^p
a_j^k(u) \p_{u_k}$ on $B_\Eucl( \eta_1)$, $j=1,\dots, p $
satisfy  $\Phi_* Z_j = \wt Y_{i_j}$ and enjoy estimate
\begin{equation}\label{cicinque}
\sup_{u\in B_\Eucl( \eta_1)}\abs{\nabla A(u)} \le   C_5;
\end{equation}

 \item [(ii)] the map $\Phi=\Phi_{I,x,r}$ is  one-to-one on the Euclidean ball
$B_\Eucl( \eta_1)$;

\item  [(iii)] for all $\eta_2\in
\left]0,\eta_1\right]$ we have  the inclusions
\begin{equation}\label{strattone}
B_\r (x_0, C_6 \eta_2^{1/s}  r  )\supseteq \Phi_{I,x_0,r} (B_\Eucl(
\eta_2))\supseteq  B_\r (x_0, C_6^{-1}\eta_2^{s}r  ),
\end{equation}
\end{enumerate*}
\end{theorem}
In Street \cite{Street}, Theorem \ref{pallascatola} was proved assuming that $Y_j\in C^2$ and that $c_{ij}^k\in C^2$. 
Here we improve the result to $Y_j\in C^1$ and $c_{ij}^k\in C^1_\O$, where $C^1_\O$ refers to $C^1$ regularity on the
manifold $\O$ described by charts of the form \eqref{lagiara}. The main novelty  is in the proof of (ii).
Namely, in Theorem
\ref{iniziare},  we use the Gronwall
inequality instead of
the uniform inverse map theorem used in
\cite[Proposition 3.20]{Street}. 
With  Theorem \ref{iniziare} in hands, the
proof of the injectivity of the map $\Phi$ is identical to the one contained in
\cite{TaoWright03}.

Since we are working with less regularity than \cite{Street}, in order to keep constants under control in terms of our data,  we give also a description of Street's arguments
to show~(i); see  Lemma
\ref{quellos} and Theorem \ref{strettino} below. 
Finally, we do not discuss the proof of (iii). Inclusion in the left-hand
side is
trivial, while the one in the  right-hand side   follows from a well known
path-lifting argument (see \cite{NagelSteinWainger,MM,Street}),
which we already used in Section \ref{distratto}.

Note that under the hypotheses of Theorem \ref{pallascatola}, possibly
shrinking
$\eta_1$, we get
\begin{equation}\label{mutti}
 \frac{\p\Phi}{\p u} = [\wt Y_{i_1,\Phi},\dots, \wt Y_{i_p,\Phi}](I_p+b(u))
\end{equation}
where $I_p+b(u) :=(I_p+a(u))^{-1}$ satisfies for some $C_4$ depending on
$L_2$ in \eqref{elledue},
\begin{equation}
 \label{muto}
\abs{b(u)} = \Bigl|\sum_{k\ge 1}(-A(u))^k\Bigr|\le
C\abs{A(u)}\le  C_4\abs{u}\qquad\text{for all $u\in B_\Eucl( \eta_1)$.}
\end{equation}

Before starting the proof of the theorem, we look at the behaviour of the
``integrability coefficients'' on a ball.
\begin{lemma}\label{quellos}
 Let $I\in \I(p, q)$ be such that  $(I, x_0, r)$ is  $\frac 12$-maximal,
with $x_0\in \Omega$ and $r\le r_0$, where  $r_0$ is small enough to ensure
that:
$B_\r(x, r_0)\subset\Omega_0$. Then we may write
\begin{equation}\label{balotelli}
 [\wt Y_i, \wt Y_j]_x= \sum_{1\le k\le p} \wt c_{ij}^k(x) \wt Y_{i_k,x}\quad
\text{for all $i,j\le q\quad x\in B_\r(x_0,\e_0 r) $,}
\end{equation}
where $\wt c_{ij}^k\in C^1_\cal{O}(B_\r(x_0, \e_0 r))$ and
\begin{equation} \label{ddop}
\max_{\substack{i,j=1,\dots,q\\k=1,\dots, p}}
\Big( \sup_{  B_{\r} (x_0, \e_0 r )}
(\abs{\wt c_{ij}^k } +
 \abs{\wt Y_\ell \wt c_{ij}^k}) \Big) \le
C=C(L_2)  .
\end{equation}
The constants $\e_0 < 1$ and $  C(L_2)>1$ depend on $L_2$ in
\eqref{elledue} but not on
$r\in (0, r_0)$.
\end{lemma}

\begin{proof}
Assume for simplicity that $I=(1,\dots,p)$. Let $\gamma$
be a Lipschitz path satisfying, a.e.~on $[0,1]$, $\dot\gamma = \sum_{j=1}^q c_j\wt
Y_j(\gamma)$. Then, arguing as in \cite[Section 4]{Street} 
(see also \cite[Proposition~3.2]{MontanariMorbidelli11d}),
 we get for a.e.~$t\in[0,1]$ the inequality
\(
\bigl|\frac{d}{dt} \Lambda_p(\gamma(t),r)\bigr|\le
C\abs{\Lambda_p(\gamma(t),r)}.
\)
Therefore,    the Gronwall's inequality \eqref{grroo} gives
$\abs{\Lambda_p(\gamma_t, r)-\Lambda_p(x,r)}\le \abs{\Lambda_p(x,r)}(e^{Ct}-1)$. Moreover we have
\begin{equation}\label{kinder}
 \abs{\wt Y_I(x)}> C^{-1} \max_{H\in \I(p, q)} \abs{\wt Y_H(x)}\quad\text{for any $x\in B_\r (x_0, \e_0 r) $.}
\end{equation}
Thus, in  the notation $I_\ell^k = (i_1, \dots, i_{k-1}, \ell,
i_{k+1}, \dots, i_p)$, by the
integrability \eqref{stop??} and
the Cramer's rule \eqref{cromo} we have 
for all $x\in B_\r(x_0, \e_0 r)$
\begin{equation*}
\begin{aligned}
 [\wt Y_i, \wt Y_j]_x  & =
\sum_{\ell=1}^q
\wh c_{ij}^\ell(x) \wt
Y_{\ell,x}
=\sum_{\ell=1}^q \wh c_{ij}^\ell(x)\sum_{k=1}^p
\frac{\langle \wt
Y_{I_\ell^k,x} ,\wt Y_{I,x}\rangle}{\abs{\wt Y_{I,x}}^2} \wt Y_{k,x}
\\&= \sum_{k=1}^p\bigg\{\sum_{\ell=1}^q \wh c_{ij}^\ell(x)  \frac{\langle \wt
Y_{I_\ell^k,x} ,\wt Y_{I,x}\rangle}{\abs{\wt Y_{I,x}}^2}\bigg \}\wt Y_{k,x}
=:\sum_{k=1}^p \wt c_{ij}^k(x) \wt Y_{k,x}.\end{aligned}
\end{equation*}
Note that  by our assumptions,  we have $\wh
c_{ij}^k \in C^1(\O)$. See the discussion after \eqref{fi}.
Moreover, since $Y_j\in C^1_\Eucl$,
 for all $j,\ell$,
we have
$\wt g_j^\ell \in C^1_\cal{O}$.
This ensures that $\wt c_{ij}^k\in C^1_\cal{O}(B_\r(x_0, \e_0 r))$ and easily
we have the
estimate $\abs{\wt  c_{ij}^k}\le C$ on $B_\r(x_0, \e_0 r)$.

Next we need to estimate the derivatives of  the coefficients $\wt c_{ij}^k$.
Note first that 
$\sup\abs{\wt Y_h  \wh c_{ij}^\ell} =r
\sup\abs{ Y_h  \wh c_{ij}^\ell}\le  r L_2\le L_2,$
 see
\eqref{elledue}.
Moreover, observe that for $x\in B_{ \r}(x_0, \e_0 r)$, $h\in\{1,\dots,q\}$,
$K\in \I(p,q)$ and $H\in \I(p,n)$, we have
\begin{equation*}
\abs{\wt Y_h \wt Y_K^H(x)}
 = \Bigl|\frac{d}{dt}\wt Y_K^H(e^{t\wt Y_h}x)\bigr|_{t=0}\Bigr|
\le C\abs{\Lambda_p(x,r)}
\le C\abs{\wt Y_I( x)}.
\end{equation*}
Here we used \eqref{kinder}.  This furnishes,  on $B_{ \r}(x_0, \e_0 r)$, the
estimate
$\Bigl|\wt Y_h \frac{\langle\wt Y_K,\wt Y_I\rangle}{\abs{\wt
Y_I}^2}\Bigr|\le C$,  for all $K\in\I(p,q)$ and $h\in\{1,\dots,q\}$.
The proof of the lemma is easily concluded.
\end{proof}

Let $(I,x,r)$ be a $\frac 12$-maximal triple for  $\P$ 
and let $\Phi= \Phi_{I,x,r}$ be the associated exponential.
For small  $\d>0$, the map $\Phi\bigr|_{B_\Eucl(
\d)}\colon
B_\Eucl( \d)\to
\Phi(B_\Eucl( \d))\subset \cal{O}$ is a $C^1$ diffeomorphism. At this
stage  there
is no control on $\delta$ in terms of the constant $L_2$ in
\eqref{elledue}.
Following \cite{TaoWright03} and  \cite{Street}, for $j\in \{1,\dots, p\}$ let
\begin{equation}
 \wh Z_j=: \sum_{k=1}^{p} \wh h_{j}^k(u) \p_{u_k}=:  \sum_{k=1}^{p}
(\d_{j}^k + \wh a_j^k(u))\p_{u_k}
\end{equation}
 be   the pull-back of  $\wt Y_{i_j}$ on the  small Euclidean ball $B_\Eucl(\d)$.
Note that $\wh a_j^k(0) = 0$. Starting from    identity
$ \sum_j {u_j\p_j} = \sum_j u_j \wh Z_j  $ on the ball $B_\Eucl( \delta)$
and commuting with  $\wh Z_i$,    one can show
 that the coefficients  $\wh a_j^k$  satisfy  the ODE
\begin{equation}
 \label{strott}
\p_\r(\r \wh A(\r\omega))= -\big\{\wh A^2(\r\omega ) +
C(\r\omega)\wh A(\r\omega) + C(\r\omega)\big\}
\end{equation}
for  $0<\r<\delta$  and $\omega\in\mathbb{S}^{p-1}$.
Here
\begin{equation}\begin{aligned}\label{cicala}
\wh A_{ik}(u)& : = \wh a_i^k(u)\quad\text{on $B_\Eucl(\d)$ \quad
and}\quad
\\C_{ik}(u)
&: = \sum_{j=1}^p
  u^j (\wt c_{ij}^k\circ \Phi)(u)\quad\text{on $B_\Eucl(\eta_1)$,}
\end{aligned}\end{equation}
where $\d>0$ is a  possibly  very  small positive number, while we may choose
$\eta_1>0$
 depending ultimately on the admissible constant $L_2$ so that
$\Phi(B_\Eucl(\eta_1))\subseteq
B_\r(x_0, \e_0 r)$.
Equation \eqref{strott} is obtained in \cite{Street}, but some
details are left to the reader. Unfortunately, we have not been able to derive
\eqref{strott} in a completely trivial way. Thus we   decided to fill
up the
details in the appendix. In particular we shall discuss all
the
regularity issues related with the fact that our vector fields $Y_j$ are
$C^1$ smooth only.

Next we give a result, which is basically  a restatement of
\cite[Theorem 3.10]{Street}.
Since we are removing some of Street's  regularity
assumptions, we do
not get  estimates on derivatives of $A$ of order greater than one.

\begin{theorem}\label{strettino} Let $\P=
\{Y_1, \dots, Y_q\}$ be a
  family  as in the assumptions of Theorem \ref{pallascatola}.  Let
$(I, x, r)$ be $\frac 12$-maximal with $x\in\Omega$ and $r\le r_0$. Denote by
$C: B_\Eucl(\eta_1)\to \R^{p\times p}$ the matrix in \eqref{cicala}.
Then, possibly taking a smaller   $\eta_1$ depending on $L_2$ in \eqref{elledue},  
there
is a unique  $A \in C^1( B_\Eucl( \eta_1), \R^{p\times p})
 $ which  solves for all $\omega\in \mathbb{S}^{p-1}$
\begin{equation}\label{disintegro}
\p_\r(\r A(\r\omega))=
  - \{A^2(\r\omega)+
C(\r\omega)A(\r\omega)+C(\r\omega) \} \quad \text{if $0<\r<\eta_1$,}
\end{equation}   satisfies
$A(0)=0 $ and enjoys the global estimate
\begin{align}\label{bottega}
  \sup_{\abs{u}\le \eta_1} \abs{\nabla A(u)}  &\le C_5,
\end{align}
where $C_5$ depends on $L_2$.
Moreover,  on the small ball
$B_\Eucl(\delta)$, we have
 $A_{jk} = \wh A_{jk}$, where $\wh A_{jk}$ is defined in \eqref{cicala}.
\end{theorem}
\begin{proof}
 We recapitulate  Street's arguments.

\step{Step 1.} By Lemma \ref{quellos} there are $\wt c_{ij}^k\in
C^1_{\cal{O}}(B_\r(x_0, \e_0 r))$ such that \eqref{balotelli}
holds  with estimate $\sup_{B_\r(x_0, \e_0 r)}\abs{\wt c_{ij}^k}\le C$, see
\eqref{ddop}.
Although at this stage, we do not have any estimate on the $C^1$ norm
$\sup_u\abs{\nabla_u( \wt
c_{ij}^k\circ\Phi)(u)}$,  we
may  use the existence part of \cite[Theorem 3.10]{Street} to obtain the
existence of a unique  $A\in C^0 (B_\Eucl(\eta_1), \R^{p\times p})$   such
that
\eqref{disintegro} holds and
$ \abs{A(u)}\le C\abs{u}$ for all $u\in B_\Eucl(\eta_1)$.

\step{Step 2.} Now we use \textit{Step 1} to  estimate the $C^1$ norm of
$(\wt c_{ij}^k\circ\Phi)$. Note first that, since   $\wt c_{ij}^k\in
C^1_{\cal{O}}(B_\r(x_0, \e_0  r)) $  and $\Phi\in C^1(B_\Eucl(\eta_1), B_\r(x_0, \e_0 r))$, we
have
 for all $1\le h,k\le p$ and $1\le i,j\le q$,
\begin{equation}\label{orizzo}
 \abs{Z_h(\wt c_{ij}^k\circ\Phi)(u)}= \abs{\wt Y_h \wt c_{ij}^k(\Phi(u))} \le
C\quad\text{for all $u\in
B_\Eucl(\eta_1)$},
\end{equation}
where the constant $C$  depends on $L_2$,   see
estimate
\eqref{ddop}.
By \textit{Step 1}, we can write for $h\in\{1,\dots,p\}$,  $Z_h = \p_{u_h} +
\sum_{j=1}^{p}a_h^j(u)\p_{u_j}$, where $\abs{a_h^j}$ is very small.
Therefore, estimate \eqref{orizzo} is equivalent to $\abs{
\nabla_u (\wt c_{ij}^k\circ\Phi)(u)}\le C $ on $B_\Eucl(\eta_1)$ for
some new
 constants $C$ and $\eta_1$ depending on $L_2$ in \eqref{elledue}.

\step{Step 3.} Here we use the hard work done in the regularity part of
\cite[Theorem 3.10]{Street} to
deduce that $A\in C^1(B_\Eucl(\eta_1), \R^{p\times p})$ and satisfies estimate
\eqref{bottega}.

\step{Step 4.} As a last step, one shows that  $A= \wh A$ on
$B_\Eucl(\d)$. This can be done as in \cite[Lemma 3.1]{Street}.
\end{proof}

In order to show the injectivity, Theorem \ref{pallascatola}-(ii),
given  a  $\frac{1}{2}$-maximal triple  $(I, x_0, r)$,  for
all $u_1\in B_\Eucl( \eta_1)$, consider the  exponential map
\begin{equation}\label{sepoffa}
 \Psi(v):=  \Psi_{u_1}(v):= \exp\Big(\sum_{1\le j\le p} v_j Z_j\Big)u_1 ,
\end{equation}
where $v$ belongs to a neighborhood of the origin in $\R^p$. The map is $C^1$,
because   $Z_j\in C^1$.

\begin{theorem}\label{iniziare}
 Let $(I, x_0, r)$ be $\frac{1}{2}$-maximal, where $x_0\in \Omega$,
$I\in\I(p_{x_0}, q)$ and $r\le r_0$. Then there is $\eta_2>0$ such that
\begin{equation}
 \frac 12 \le \frac{\abs{\Psi_{u_1}(v) - \Psi_{u_1}(\ol v)}}{\abs{v-\ol v}}\le
2\quad\text{for all $u_1\in B_\Eucl( \eta_2)\quad v,\ol v\in
B_{\Eucl}(\eta_2)$.}
\end{equation}
\end{theorem}
Note that Theorem \ref{iniziare} implies that  for all $u_1\in
B_\Eucl(\eta_2)$, the map  $\Psi_{u_1}$ is one-to-one on $B_\Eucl(\eta_2)$
and, by a  standard path-lifting argument, it ensures  the   quantitative 
openness condition
 $\Psi_{u_1}(B_\Eucl(\eta_2))\supset B_\Eucl\big(u_1, \frac{1}{2}\eta_2\big)$
for all $u_1\in B_\Eucl(\eta_2)$.

Once Theorem \ref{iniziare} is proved, then the injectivity of the map
$\Phi$ follows from the  argument in \cite[p.~622]{TaoWright03}, or
\cite[Proposition 3.20]{Street}. We omit the proof.

\begin{proof}[Proof of Theorem \ref{iniziare}] It suffices to   show that
there is
$\eta_2\leq \eta_1$ such that for all $u_1\in B_\Eucl( \eta_2)$,  the map
$\Psi=\Psi_{u_1}$ satisfies    
\begin{equation*}
\sup_{v\in B_\Eucl( \eta_2)}\abs{d\Psi_{u_1}(v) - I_p} \le \frac 12
\quad\text{for all $u_1\in B_\Eucl( \eta_2)$,}
\end{equation*}
where as usual $\abs{\cdot}$ denotes the operator norm.

To show this estimate, recall that the vector
fields
  $Z_j= \p_{j} + \sum_{k} a_j^k(u)\p_k$ on $B_\Eucl( \eta_1)$ satisfy
\eqref{cicinque}. Therefore,
\begin{equation}
 \label{larghetto}
\abs{a(u)}\le C_5 \abs{u} < \eta_1,
\end{equation}
provided that  $\abs{u} < \eta_1/C_5$.
Now we show that
\begin{equation}
\label{zappa3}
\abs{u_1}<\dfrac{\eta_1}{2C_5}
 \quad\text{and}\quad    \abs{v}     <\dfrac{\eta_1}{4C_5}
 \quad\Rightarrow \quad \abs{\Psi_{u_1}(v)} < \frac{\eta_1}{C_5}.
\end{equation}
To prove \eqref{zappa3} let $y = y(t,v): = \Psi_{u_1} (t v)$.  Assume that
for some $t_0 \le 1$ we have
\begin{equation*}
 \frac{\eta_1}{C_5}= \abs{y(t_0,v)}>\abs{y(t,v)}\quad\text{for all $t\in
 \mathopen{[}0,t_0\mathclose{[}$.}
\end{equation*}
Then
\begin{equation*}
\begin{aligned}
 \frac{\eta_1}{C_5}& = \abs{y(t_0)}\le
\abs{u_1}+\Bigl|\int_0^{t_0}\big(I_p + a(y(\t))\big)v d\t\Bigr|
\le
\abs{u_1} +\abs{v} t_0 + \eta_1\abs{v}t_0
\\&< \frac{\eta_1}{2C_5} + \frac{\eta_1}{2C_5}t_0.
\end{aligned}
\end{equation*}
But this can not hold unless $t_0>1$. Therefore, \eqref{zappa3} is proved.

Let us look again at $y=y(t, v) $, note that $\frac{\p y^k}{\p v_j}(0,v)=0$ and
$\abs{a(y(t,v))}< \eta_1$ for all $t\in[0,1]$, if $\abs{u}< \frac{\eta_1}{2
C_5}$ and
$\abs{v}< \frac{\eta_1}{4 C_5}$ (this follows from \eqref{larghetto} and
\eqref{zappa3}).
Write the variational equation
\begin{equation*}
\begin{aligned}
 \frac{d}{dt}\frac{\p y^k}{\p v_j}&=\frac{\p}{\p v_j}
\Big(\sum_{1\le \ell\le p}(\d_\ell^k +
a_\ell^{k}(y)) v_\ell \Big)
=\d_j^k +
a_j^{k}(y) +  \sum_{1\le \ell,h\le p}  \p_h a_\ell^{k}(y)\frac{\p y^h}{\p
v_j} v_\ell.
\end{aligned}
\end{equation*}
Denote $\big(\frac{\p y^k}{\p v_j}(t)\big)_{j,k=1}^p =: w(t)\in \R^{p\times p}$
and
$
 (L_v(t))_h^k:= \sum_{\ell=1}^p \p_h a_\ell^{k}(y(t))v_\ell
$.
Note estimate $\abs{L_v(t)}\le C_5\abs{v}$.
Starting from the  ODE $\dot
w(t) = I_p + a(y(t)) + L_v(t) w(t)$ and integrating, we obtain
\begin{equation*}
\begin{aligned}
 \abs{w(t)-tI_p} & \le \int_0^t \big\{
C_5  \abs{v}\,\abs{w(\t)- \t I_p} +
\t \abs{L_v(\t)} + \abs{a(y(\t))}\big\}d\t
\\&\le C_5\abs{v}\int_0^t  \abs{w(\t)- \t I_p} d\t +
C_5 t\abs{v}+t\eta_1,
\end{aligned}
\end{equation*}
for all $t\in[0,1]$. The Gronwall inequality \eqref{grroo}
gives
\begin{equation*}
 \abs{w(1) - I_p}\le
\frac{C_5\abs{v}+\eta_1}
{C_5\abs{v}}\bigl(\exp(C_5\abs{v})
-1\bigr)\le
\frac 12,
\end{equation*}
as soon as we  assume without loss of generality that  $\eta_1<\frac14$
and we take $\abs{v}\le
\eta_2$
where $\eta_2$ is small enough, depending on $C_5$ and $\eta_1$.
\end{proof}

\section{Some remarks in  a more regular setting} \label{Sttt}

In more regular situations than ours, the Poincar\'e inequality can be obtained putting together
 Jerison's result and Street's theorem, 
\cite{Jerison} and \cite{Street}. 
In this section we briefly  discuss this idea and we explain why our regularity classes fall out of this approach.
The discussion included here has been suggested by the referee.

Let $\H:= \{X_1, \dots, X_m\}$ be a smooth family and assume that the family  $\P=\P_s= \{Y_1, \dots, Y_q\}$
satisfies \eqref{intello} for smooth coefficients $c_{ij}^k$. By \cite{Street}, there is $C_\flat>1$ such that 
if $(I,x,r)$ is $\frac 12$-maximal ($I\in\I(p,q)$), then  the associated map   $\Phi: = \Phi_{I, x, r}$  is one-to-one on $B_\Eucl(C_\flat^{-1})\subset\R^p$ and satisfies
\begin{align*}
C_\flat^{-1}\abs{\Lambda_p(x,r} &\le \abs{J_\Phi(u)}\le C_\flat\abs{\Lambda_p(x,r}\quad\text{for all $u\in B_\Eucl(C_\flat^{-1})$,}
\\& B_\r (x, C_\flat^{-2}r)\subset \Phi(B_\Eucl(C_\flat^{-1})).      
\end{align*}
Moreover, letting $Z_j:=\Phi^*(r X_j)$ be the pullbacks of the 
vector fields
and taking   $f \in C^1(\R^n,\R) $, we have 
  for any $a\in\R$,
\begin{equation*}
      \int_{B_\cc(x, C_\flat^{-2} r)} \abs{f(y)-a}d\s^p(y) \le C_\flat \abs{\Lambda_p(x,r)}
\int_{B_Z(0, C_\flat^{-2})}\abs{g(u)-a}du.
\end{equation*}
where $g=f\circ\Phi$.
By the Jerison's Poincar\'e inequality, since $Z_1, \dots, Z_m$ are smooth  H\"ormander
vector fields in $B_\Eucl(C_\flat^{-1})$, there is $C_\sharp>1$ so that
\begin{equation}\label{lapara} 
      \inf_a\int_{B_Z(0, C_\flat^{-2})}\abs{g(u)-a}du\le C_\sharp   \int_{B_Z(0, C_\flat^{-2})}
\sum_{1\le j\le m}\abs{Z_jg(u)}du.
\end{equation}
Going back to the variable $y$, one gets the Poincar\'e inequality for the original  smooth vector fields $X_j$.

Let us make some comments on the constants $C_\flat,  C_\sharp$ appearing in the computations above. 
By Street's theorem \cite{Street}, we have  
$C_\flat =C_\flat (\norm{Y_j}_{C^2}, \|{c_{ij}^k}\|_{C^2} ).$
  The constant $C_\sharp$  appears in Jerison's paper.  It is known that
$
      C_\sharp = C_\sharp (\|X_j\|_{C^M}), 
$
where $M$ is a rather large number for which there is no precise estimate.

Regularity requirements in the argument above   can be improved avoiding the Jerison's proof of the  Poincar\'e 
inequality and using the approach  
in \cite{MM}. Indeed, in \cite{MM}, using almost exponential maps, it was proved
that  \eqref{lapara} holds with 
\[
     C_\sharp =
  C_\sharp \Bigl (\|Z_j\|_{C^{s-1,1}}, \inf_{B_\Eucl(C_\flat^{-1})} \max_{\abs{w_1}, \dots, \abs{w_p}\le s}
\abs{\det (Z_{w_1}, \dots, Z_{w_p})} \Bigr),
\]
where    $Z_{w_j}$ is a commutator of $Z_1, \dots, Z_m$ of  length $\abs{w_j}$.
By  \cite{Street}, the infimum can be controlled similarly to $C_\flat$,  
in terms of  $\norm{Y_j}_{C^2}, \|{c_{ij}^k}\|_{C^2} $.
In order to ensure that $Z_j:= \Phi^*(r X_j)\in C^{s-1,1}$, it suffices 
to assume that $\Phi\in C^{s,1}$ and $X_j\in C^{s-1,1}$.
In view of the explicit  form  $\Phi(u) = \exp\textstyle{\bigl(\sum_{k=1}^p u_k Y_{i_k}\bigr)(x)}$, a
 concrete sufficient condition  on the original vector fields $X_1, \dots, X_m$ to ensure that $\Phi\in C^{s,1}$ 
is the assumption
 $Y_1, \dots, Y_q\in C^{s,1}$. 
\footnote{Note that the $Y_j$'s involve derivatives up to order $s-1$. Therefore an easy-to-state assumption to ensure that $Y_j\in C^{s,1}$ is   $X_j\in C^{2s-1,1}$.}
This is a stronger assumption than ours, which requires that only the original vector fields  $X_1, \dots, X_m$ belong to $  C^s$.

\appendix
\section{Appendix}


Here we   discuss a detailed derivation of \eqref{strott}, in which
we use the
fact that any orbit $\cal{O}_\cal{P}$ associated with $\cal{P}= \{Y_1, \dots,
Y_q\}$,  is   a $p$-dimensional $C^2$ immersed submanifold of $\R^n$.
 Since we are discussing a regularity issue,  without loss of generality
we may assume that $r=1$ so that no tilde symbols appear.

Recall that given a    $C^2$ manifold $\O$, we say that $U$
is a $C^1$ vector field on $\O$
if in any $C^2$ coordinate system $\O\supset\Omega\ni x\mapsto\a(x)=\xi\in
\a(\Omega)\subset\R^p $, we have  $U_x = \sum  U^j(x)\big
(\frac{\p}{\p \xi_j}\big)_x$ for all $x\in \Omega$, where $U^j = U\a^j$
is a $C^1$ function on $\Omega$.
\begin{remark}\label{varieta} We recall some known facts  about $C^2$ manifolds.
\begin{enumerate}[noitemsep,label=(\alph*)]
 \item  The notion of $C^1$ vector field
is well defined
(coordinate invariant) provided that $\O$ is at least $C^2$.

\item Integral curves of a $C^1$ vector field on a $C^2$ manifold $\O$ are
unique and the map $x\mapsto e^{tU}x$ is
$C^1$~smooth. Indeed, a path $t\mapsto \gamma(t)\in\Omega$
 is an integral curve
of $U$ if and only if $\a\circ\gamma $ is an integral curve of the vector
field $\sum_{k}(U^k\circ\a^{-1})(\xi)\frac{\p}{\p\xi_k}$ which is a $C^1$ vector
field in  $\a(\Omega)$.
\item If $U$ and $V$ are $C^1$ vector fields in $\Omega\subseteq \O$,
one can check that the commutator
$
 [U, V]_x:= \sum_j (U V^j(x)  - V U^j(x)) (\p_{\xi^j})_x
$
is well defined independently on the coordinate system and it turns out
that $[U, V] = \cal{L}_U V$.
Finally, if
$U,V$ are $C^1$ vector fields and
$\Psi\in C^1(\Omega)$, then
\begin{equation}\label{lieto}
\begin{aligned}
 \cal{L}_u V  & = [U, V] = \sum_{j}\{ U V^j(x) - V
U^j(x)\}\Bigl(\frac{\p}{\p
\xi^j}\Bigr)_x \quad\text{and}
\\& [U, \Psi V]= U\Psi V +\Psi[U, V]
\end{aligned}
\end{equation}
\item If $\O$ is a $C^2$ submanifold of $\R^n$ and $Y$ is a $C^1$ vector
field in $\R^n $ which satisfies $Y_x\in T_x \O$ for all $x\in \O$, then  any integral curve of $Y$ starting
from $\O$ can not leave~$\O$ for small times.
\end{enumerate}
 \end{remark}
All  items (a),(b) and (c)  can be checked relying on the fact that the
coordinate
versions of a $C^1$ vector field on a $C^2$ manifold $\O$ are $C^1$.
 Statement (d) is related with the embedding of $\O$ in $\R^n$ and can be
checked
for instance by writing  a $C^2$ local change of coordinates in $\R^n$ which
makes
$\O$
of the form $\R^p\times \{0\}$. This standard argument 
works well as soon as the manifold is $C^{1,1}$ at least. In less regular cases  
one can use Bony's theorem. 

Next we come to the derivation of \eqref{strott}. Let $W:=\sum_{j=1}^p
u_j\p_{u_j}$.
Start from identity
\begin{equation}\label{giacc}
\begin{aligned}
\Phi_*\Bigl( \sum_{j=1}^p u_j \p_{u_j}\Bigr)
& =
\frac{d}{d\e}\Phi((1+\e)u)\bigr|_{\e=0}
= \sum_{j=1}^p u_j
 Y_{j,\Phi(u)}.
\end{aligned}
\end{equation}
Since $\O$ is a $C^2$ manifold and
$T_x\O = \Span\{Y_{j,x}:j=1,\dots, p\} $ for all $x= \Phi(u)$, Remark
\ref{varieta}, (d)  ensures that $\Phi(B_\Eucl(\delta))\subset\O$.
The map $\Phi|_{B_\Eucl(\d)}:B_\Eucl(\d)\to
\Phi(B_\Eucl(\d)) \subseteq\cal{O}$ is a $C^1$ diffeomorphism and its inverse
    $\Psi= \Phi^{-1} $ can be used as a $C^1$ chart.  Then,
at any  $x= \Phi(u)$ with  $\abs{u}\le \delta$ we have
\begin{equation}\label{vectoroz}
(\Phi_*W)_x = \sum_j\Psi^j(x) Y_{j,x}
\end{equation}
Observe that the a priori continuous vector field $\Phi_*W$ is
actually $C^1$. This follows looking at the right-hand side of
\eqref{vectoroz}. Indeed,  $\Psi^j$ is a $C^1 $ function and,
by Remark \ref{varieta}-(d),
$Y_j$ is a $C^1$ vector field
(in both statements $C^1$ refers to the $C^2$ differential
structure of   $\O$ described in \eqref{lagiara}).
Thus its integral curves
are unique
and the flow $x\mapsto
e^{-t \Phi_*W} x$ is a  $C^1$  local  diffeomorphism on $\cal{O}$.
See Remark \ref{varieta}-(b).

Next, note that, if $\delta>0$ is small enough and
$\abs{u}\le \delta$, then  the linear system
$D\Phi(u)h_j(u)=   g_j(\Phi(u))$
has a unique solution $h_j(u)\in \R^p$ (here $D\Phi(u)\in \R^{n\times p}$
denotes the Jacobian matrix). The solution $h_j(u)$ is given by the Cramer's rule
 \eqref{cromo}.
 Let $Z_{j,u}:= h_j(u)\cdot\nabla$ be the
corresponding continuous vector field. Pulling back  \eqref{giacc},  we get
$
 \sum_j u_j\p_j = \sum_j u_j Z_j
$, see \cite{TaoWright03,Street}.

We claim that $u\mapsto W^\sharp h_j(u):=\cal{L}_W Z_j(u)$ is a continuous function in
$B_\Eucl(\delta)$.
By the
Cramer's rule \eqref{cromo}, this claim follows from
the continuity of $u\mapsto W^\sharp (\p_j\Phi)(u)$, which will be
checked in Lemma
\ref{laodi} below, and from the continuity of $g_j\circ\Phi$.

Since $W=\sum_{j=1}^pu_j\partial_{u_j}$ then  $e^{tW} u = e^{t} u$ and
 this gives   the
expansion
$h_i(e^tu)=h_i(u)+W^\sharp h_i(u)t(1+o(1))$, as $t\to 0$ and
 \begin{align*}
 \cal{L}_W Z_i(u)&
= \lim_{t\to 0}\frac 1t
\big\{e^{-t}h_i (e^{t}u)  - h_i(u)\big\}
= - h_i(u) + W^\sharp h_i(u).
\end{align*}

  Let now  $\eta= \sum_{k=1}^{p} \eta_k(u) du_k$
be a smooth $C^\infty_c$ one form.
Testing $\cal{L}_WZ_j$ against   $\eta$, we get
\begin{equation}\label{aojj}
\begin{aligned}
\langle \cal{L}_W Z_i,  \eta\rangle
&=\int \sum_k\Big(W^\sharp h_i^k(u)\eta_k(u) - h_i^k(u)\eta_k(u)\Big)du
\\&
  = \Big\langle \sum_{k,j}
 u_j (D_{u_j} h_i^k)\p_k ,\eta\Big\rangle - \Big\langle\sum_{k}
h_i^k(u)\p_k,\eta\Big\rangle ,
\end{aligned}
\end{equation}
where $D_{u_j} h_i^k\in \mathcal{D}'$ denotes the distributional
derivative of the continouos function $h_i^k$. Equality \eqref{aojj} can be
checked using the definition of $W^\sharp:= \cal{L}_W$ and integrating by parts.

Next we want to write   $\cal{L}_W Z_i$ in a different way, in order to
use the integrability condition.  To this aim, we calculate  its
push forward. Let $\Phi_*: T B_\Eucl(\d)\to T \cal{O}$ be the tangent
map.  Fix $u\in B_\Eucl(\delta)$ and  let $x=\Phi(u)$. Then,
\begin{equation*}\begin{aligned}
\Phi_*(\cal{L}_W Z_i)_u &: = \Phi_*\lim_{t\to 0}\frac{1}{t}
\big\{ e^{-tW}_* (Z_{i,e^{tW}u}) - Z_{i,u}\big\}
\\&=
\lim_{t\to 0}\frac 1t\big\{(\Phi\circ e^{-tW})_* (Z_{i,e^{tW}u} ) -
Y_{i,x})\big\},
\end{aligned}
\end{equation*}
because  $\Phi:B_\Eucl(\d)\to \cal{O}$ is $C^1$, so that $\Phi_* e^{-t W}_*=
(\Phi\circ e^{-t W})_*$.
Since  the function
$e^{-t\Phi_*W}$ is  the flow of a $C^1_\cal{O}$ vector field, it is
$C^1_\cal{O}$, see Remark \ref{varieta}-(b).
Therefore,
$
\Phi\circ e^{-tW} = e^{-t \Phi_*W}\circ \Phi$ and
  we have
\begin{align*}
 (\Phi\circ e^{-t W})_* Z_{i, e^{tW}  u}
&= (e^{-t\Phi_*W}\circ\Phi)_* Z_{i, e^{tW}  u}=
e^{-t\Phi_*W}_* \Phi_* Z_{i, e^{tW}  u}
=
e^{-t\Phi_*W}_* Y_{i, e^{t\Phi_*W}\Phi(u)}.
\end{align*}
We have shown that
$\displaystyle \Phi_*\cal{L}_W Z_i = \cal{L}_{\Phi_* W}\Phi_*Z_i  =
\cal{L}_{\sum_j \Psi^j Y_j} Y_i$ under our regularity assumptions (this
is a well known fact for smooth
vector fields).
Since the vector field $\sum\Psi^j Y_j$ is $C^1$ on $ \cal{O}$,  by
\eqref{lieto},  we may write
 \begin{equation*}
\begin{aligned}
 \Phi_*\cal{L}_W Z_i
 &  = \cal{L}_{\sum \Psi^j Y_j} Y_i = \Big[ \sum _j \Psi^j Y_j, Y_i \Big]
 =- \sum_j Y_i\Psi^j  Y_j - \sum_{j}\Psi^j [Y_i, Y_j]
\\&=-  \sum_j Y_i\Psi^j  Y_j - \sum_{j,k} \Psi^j c_{ij}^k Y_k.
\end{aligned}
\end{equation*}
Pulling  back, we get
\begin{equation}\label{bojj}
\begin{aligned}
\cal{L}_W Z_i &=\Phi_*^{-1} \Phi_*\cal{L}_W Z_i
= -\sum_j h_i^j(u)Z_j - \sum_{j,k} u_j (c_{ij}^k\circ\Phi)
Z_k .
\end{aligned}
\end{equation}
Here we used the equality $Y_i\Psi^j = h_i^j(u)$. \footnote{This  can
be proved as follows. Possibly choosing a smaller $\delta$, we may extend $\Psi$
to a $C^1$ function $\ol\Psi$ defined in a open set in $\R^n$ containing
$\Phi(B_\Eucl(\d))$. Then, 
\begin{equation*}
\begin{aligned}
Y_i\Psi^j(x) & =
\sum_{\a=1}^n \p_\a\ol \Psi^j(\Phi(u)) g_i^\a(\Phi(u))
 =\sum_{\a=1}^n \p_\a\ol\Psi^j(\Phi(u)) \sum_k
h_i^k(u)\p_k\Phi^\a(u) = h_i^j(u),
\end{aligned}
\end{equation*}
because $\ol \Psi\circ\Phi(u)= u$ for all $u$.
}

We have obtained two different expressions for $\cal{L}_W Z_i$, namely
\eqref{aojj} and \eqref{bojj}. In order to compare them, it suffices to
test \eqref{bojj} against $\eta$. This gives
the distributional identity
\begin{equation*}
 \sum_{k,j} u_j D_j h_i^k\p_k - \sum_jZ_i u_j \p_j = -\sum_j Z_i u_j Z_j -
\sum_{j,k} u_j (\wt c_{ij}^k\circ\Phi)Z_k,
\end{equation*}
where $Z_i u_j = h_i(u)\cdot \nabla u_j = h_i^j(u)$. Last equality is
 exactly  formula
(3.5) in \cite{Street}. From now on, it suffices
to follow Street's calculations and we get the ODE \eqref{strott}.

\begin{lemma}
\label{laodi} Let $Y_j:= g_j\cdot\nabla $, where  $ g_j\in C^1_\Eucl$ for
$j=1,\dots,p$. Let
$\Phi$ be the exponential map in \eqref{fi} and let
$\Phi_j:= \p_j\Phi$. Then the map $u\mapsto W^\sharp \Phi_j(u):= \lim_{\e\to
0}\frac{1}{\e}(\Phi_j(e^\e u)- \Phi_j(u))$ is continuous in a neighbourhood of
the origin.
\end{lemma}
\begin{proof} Let $\eta(t,u)=\Phi(tu)$ be the solution of
$\frac{\p\eta}{\p t}(t,u) = \sum_{j=1}^p{u_j g_j(\eta(t,u))}$ with
$\eta(0,u)=x_0$.
 Since $e^{\e W}0=0$ for all $\e$, we  have
$W^\sharp \Phi_j(0)= 0$.

In order to calculate $W^\sharp \Phi_j(u)$ for $u\neq 0$, note that
$ \frac{\p\eta }{\p u_j}(t,u)=
t\Phi_j(tu)$, for any $t$ and $u$ close to  $0$. Therefore,
if $u\neq 0$, we have
\begin{equation}\label{calcolino}
\begin{aligned}
 W^\sharp \Phi_j(u)&:
=\lim_{t\to 1}\frac{1}{t-1}
\bigl( \Phi_j(t u )-\Phi_j(u)\bigr)
\\&=\lim_{t\to 1}\frac{1}{t-1}
\Bigl(\frac 1t\frac{\p\eta }{\p u_j}(t,u)- \frac{\p\eta }{\p
u_j}(1,u)\Bigr).
\end{aligned}
\end{equation}
But  the definition of partial derivative and the variational
equation give
\begin{equation*}
\begin{aligned}
 &\lim_{t\to 1}\frac{1}{t-1} \Bigl(\frac{\p\eta}{\p u_j}(t,u) -
\frac{\p\eta}{\p u_j}(1,u) \Bigr)
= \frac{\p^2\eta}{\p t\p u_j} (1,u)
= \frac{\p}{\p u_j} \sum_k u_k g_k(\eta(1,u)).
\end{aligned}
\end{equation*}
In other words, since $\eta(1,u)=\Phi(u)$,
\begin{equation*}
 \frac{\p\eta}{\p u_j} (t,u)  = \frac{\p\eta }{\p u_j} (1,u) + (t-1)
\Bigl(g_j(\Phi(u))+\sum_{k,i} u_k\p_i g_k(\Phi(u))\Phi_j^i(u)\Bigr)(1+o(1)),
\end{equation*}
which, inserted into \eqref{calcolino}, gives
$ W^\sharp\Phi_j(u)  =
 - \Phi_j(u) + g_j(\Phi(u)) +
 \sum_{k,i} u_k\p_i g_k(\Phi(u))\Phi_j^i(u).$
This shows that $W^\sharp\Phi_j$ is a continuous function at any  $u\neq 0$.
Moreover,
 $W^\sharp\Phi_j(u)\to - \Phi_j(0)+ g_j(\Phi(0)) = 0 $, as $u\to 0$. Since
we already claimed that $W^\sharp \Phi_j(0)=0$, the proof is concluded.
\end{proof}

\footnotesize

 \phantomsection
\addcontentsline{toc}{section}{References}  



\def\cprime{$'$}
\providecommand{\bysame}{\leavevmode\hbox to3em{\hrulefill}\thinspace}
\providecommand{\MR}{\relax\ifhmode\unskip\space\fi MR }
\providecommand{\MRhref}[2]{%
  \href{http://www.ams.org/mathscinet-getitem?mr=#1}{#2}
}
\providecommand{\href}[2]{#2}

\normalsize
\bigskip \noindent\sc \small  Annamaria Montanari, Daniele Morbidelli
\\ Dipartimento di Matematica,
Universit\`{a} di Bologna  (Italy)
\\Email: \tt   annamaria.montanari@unibo.it,
daniele.morbidelli@unibo.it

\end{document}